\documentclass[final]{ws-ccm}

\usepackage{mathtools}

\newcommand{\ds}{\displaystyle}
\newcommand{\Nb}{{\mathbb{N}}}
\newcommand{\Rb}{{\mathbb{R}}}
\newcommand{\Zb}{{\mathbb{Z}}}
\newcommand{\A}{{\mathcal{A}}}
\newcommand{\D}{{\mathcal{D}}}
\newcommand{\R}{{\mathcal{R}}}
\newcommand{\C}{{\mathcal{C}}}

\newcommand{\F}{{\mathcal{F}}}
\newcommand{\SN}{{\mathbb{S}^{N-1}}}
\newcommand{\Sd}{{\mathbb{S}^{d-1}}}

\newcommand{\LL}{{\mathcal{L}}}
\newcommand{\HH}{{\mathcal{H}}}
\newcommand{\M}{{\mathcal{M}}}

\newcommand{\med}{- \hskip -1em \int}

\newcommand{\res}{\mathop{\hbox{\vrule height 7pt width .5pt depth 0pt
\vrule height .5pt width 6pt depth 0pt}}\nolimits}

\def\hom{\text{hom}}
\def\dist{\text{dist}}

\let\e=\varepsilon
\let\O=\Omega

\let\G=\Gamma
\let\a=\alpha
\let\b=\beta
\let\d=\delta
\let\g=\gamma

\begin{document}

\markboth{J.-F. Babadjian \& V. Millot} {Homogenization of
variational problems in manifold valued $BV$-spaces}

\title{HOMOGENIZATION OF VARIATIONAL PROBLEMS\\ IN MANIFOLD VALUED $BV$-SPACES}
\author{Jean-Fran\c{c}ois BABADJIAN\footnote{{\it Current adress:} CMAP, Ecole Polytechnique, 91128 Palaiseau, France. 
{\it E-mail:} \texttt{babadjian@cmap.polytechnique.fr}}}

\address{Laboratoire Jean Kuntzmann\\
Universit\'e Joseph Fourier\\
BP 53\\
38041 Grenoble Cedex 9, France.\\
\emph{\tt{babadjia@imag.fr}}}

\author{Vincent MILLOT}

\address{Univesit\'e Paris Diderot - Paris 7\\
CNRS, UMR 7598 Laboratoire Jacques-Louis Lions\\
F-75005 Paris, France.\\
\emph{\tt{millot@math.jussieu.fr}}}

\maketitle

\begin{abstract}
{\bf Abstract.} This paper extends the result of \cite{BM} on the
homogenization of integral functionals with linear growth defined for
Sobolev maps taking values in a given manifold. Through a $\Gamma$-convergence analysis, we identify the homogenized energy
in the space of functions of bounded variation. It turns out to be finite for $BV$-maps with values in the manifold. The bulk and Cantor parts of the
energy involve the tangential homogenized density introduced in
\cite{BM}, while the jump part involves an homogenized surface density given by a geodesic type problem on the manifold.
\end{abstract}





\section{Introduction}

\noindent In this paper we extend our previous resut \cite{BM} concerning the
homogenization of integral functionals with linear growth involving manifold valued mappings.
More precisely, we are interested in energies of the form
\begin{equation}\label{mainfunct}\int_\O f\left(\frac{x}{\e},\nabla
u\right)dx\,,\quad u : \O \to \mathcal{M}\subset\Rb^d\,,\end{equation}
where $\O \subset \Rb^N$ is a bounded open set,
$f:\Rb^N \times \Rb^{d \times N} \to [0,+\infty)$ is a periodic integrand
in the first variable with linear growth in the second one, and $\M$ is a smooth submanifold. Our main goal is to
find an effective description of such energies as $\varepsilon\to 0$. To this aim we perform a $\Gamma$-convergence analysis
which is an appropriate approach to study asymptotics in variational problems (see \cite{DM} for a detailed description of this subject).
For energies with superlinear growth, the most general homogenization result has been obtained independently in \cite{Br,M} in the nonconstrained case, and in
\cite{BM} in the setting of manifold valued maps.

The functional  in \eqref{mainfunct} is naturally defined for maps in the Sobolev class $W^{1,1}$. However if one wants to apply
the Direct Method in the Calculus of Variations, it becomes necessary to extend the original energy to a larger class of functions (possibly singular) in which the existence
of minimizers is ensured. In the nonconstrained case, this class is exactly the space of functions of bounded variation and the problem of finding an
integral representation for the extension, the so-called {\it ``relaxed functional"},  has been widely invetigated, see {\it e.g.}, \cite{Serr,GMSbv,DMsc,AMT,AmPal,FR,ADM,FM,FM2,BFF}
and \cite{Bouch,DAG} concerning homogenization in $BV$-spaces.

Many models from material science involve vector fields taking their values into a manifold.
This is for example the case in the study of equilibria
for liquid crystals, in ferromagnetism or for magnetostrictive
materials.  It then became necessary to understand the
behaviour of integral functionals of the type (\ref{mainfunct})
under this additional constraint.
In the framework of Sobolev spaces, it was the object of
\cite{DFMT,AL,BM}. For $\varepsilon$ fixed, the complete analysis in the linear growth case has been performed in
\cite{AEL} assuming that the manifold is the unit sphere of $\Rb^d$.
Using a different approach, the arbitrary manifold case has been recently treated in
\cite{Mucci} where a further isotropy assumption on
the integrand is made. We will present in the Appendix the analogue result to \cite{AEL}
for a general integrand and a general manifold.

We finally mention that the topology of $\M$ does not play an important role here. This is in
contrast with a slightly different problem originally introduced in \cite{BCL,BBC}, where the starting
energy is assumed to be finite only for smooth maps. In this direction,
some recent results in  the linear growth case can be found in \cite{GM,GM1} where the study
is performed within  the framework of Cartesian Currents \cite{GMS}.
When the manifold $\M$ is topologically nontrivial, it shows the emergence in the relaxation process of non
local effects essentially related to the non density of
smooth maps (see \cite{B,BZ}).
\vskip5pt

Throughout this paper we consider a
compact and connected smooth submanifold $\M$ of $\Rb^d$ without
boundary. The classes of maps we are interested in are defined as
$$BV(\O;\M):=\big\{ u \in BV(\O;\Rb^d) : \; u(x) \in \M \text{ for $\LL^N$-a.e. }x \in \O\big\}\,,$$
and $W^{1,1}(\O;\M)=BV(\O;\M) \cap W^{1,1}(\O;\Rb^d)$. For a smooth $\M$-valued map, it is well known that first order derivatives belong to the tangent  space of $\M$, and this
property has a natural extension to $BV$-maps with values in $\M$, see Lemma \ref{manifold}.
\vskip5pt

The function $f : \Rb^N \times \Rb^{d \times N} \to [0,+\infty)$ is
assumed to be a Carath\'eodory integrand satisfying
\begin{itemize}
\item[$(H_1)$] for every $\xi \in
\Rb^{d \times N}$ the function $f(\cdot,\xi)$ is $1$-periodic, {\it
i.e.} if $\{e_1,\ldots,e_N\}$ denotes the canonical basis
of $\Rb^N$, one has $f(y+e_i,\xi)=f(y,\xi)$ for every $i=1,\ldots,N$ and $y \in \Rb^N$;\\
\item[$(H_2)$] there exist $0<\a \leq \b < +\infty$
such that
$$\a |\xi| \leq f(y,\xi)\leq \b(1+|\xi|) \quad \text{ for a.e. }y \in \Rb^N \text{ and all }
\xi \in \Rb^{d \times N}\,;$$
\item[$(H_3)$]there exists $L>0$ such that
$$|f(y,\xi)-f(y,\xi')| \leq L |\xi-\xi'|\, \quad \text{ for a.e. }y \in \Rb^N \text{ and all }
\xi,\, \xi' \in \Rb^{d \times N}\,.$$
\end{itemize}
For $\e>0$, we define the functionals $\F_\e:L^1(\O;\Rb^d) \to
[0,+\infty]$ by
$$\F_\e(u):=\begin{cases} \ds \int_\O f\left(\frac{x}{\e},\nabla
u\right) dx & \text{if }u \in
W^{1,1}(\O;\mathcal{M})\,,\\[8pt]
+\infty & \text{otherwise}\,.
\end{cases}$$
\vskip5pt

We have proved in \cite{BM} the following representation result on
$W^{1,1}(\O;\M)$.

\begin{theorem}[\cite{BM}]\label{babmilp=1}
Let $\M$ be a compact and connected smooth submanifold of $\Rb^d$
without boundary, and  $f:\Rb^N \times \Rb^{d \times N} \to
[0,+\infty)$~be a Carath\'eodory function satisfying $(H_1)$ to
$(H_3)$. Then the family $\{\F_\e\}_{\e>0}$ $\G$-converges for the
strong $L^1$-topology at every $u \in W^{1,1}(\O;\M)$ to $\F_{\rm
hom} : W^{1,1}(\O;\M) \to [0,+\infty)$, where
$$\F_{\rm hom}(u):= \int_\O Tf_{\rm hom}(u,\nabla u)\, dx\,,$$
and $Tf_{\rm hom}$ is the tangentially homogenized energy density
defined for every $s\in \M$ and $\xi\in [T_s(\M)]^N$ by
\begin{equation}\label{Tfhom}
Tf_{\rm hom}(s,\xi)=\lim_{t\to+\infty}\inf_{\varphi} \bigg\{
\med_{(0,t)^N} f(y,\xi+ \nabla \varphi(y))\, dy :  \varphi \in
W^{1,\infty}_0((0,t)^N;T_s(\mathcal{M})) \bigg\}.
\end{equation}
\end{theorem}
\vskip5pt

Note that the previous theorem is not really satisfactory since the domain of the $\G$-limit is obviously larger than the Sobolev
space $W^{1,1}(\O;\M)$. In view of the studies performed in
\cite{GM,Mucci}, the domain is exactly given by  $BV(\O;\M)$. 
Under the additional (standard)
assumption,
\begin{itemize}
\item[$(H_4)$]there exist $C>0$ and $0<q<1$ such that
$$|f(y,\xi)-f^\infty(y,\xi)| \leq C (1+|\xi|^{1-q}) \quad \text{ for a.e. }y \in \Rb^N \text{ and all }
\xi \in \Rb^{d \times N}\,,$$ where $f^\infty:\Rb^N \times \Rb^{d
\times N} \to [0,+\infty)$ is the recession function of $f$ defined
by
$$f^\infty(y,\xi):=\limsup_{t \to +\infty}\,\frac{f(y,t\xi)}{t}\,,$$
\end{itemize}
we have extended Theorem~\ref{babmilp=1} to
$BV$-maps, and our main result can be stated as follows.

\begin{theorem}\label{babmil2}
Let $\M$ be a compact and connected smooth submanifold of $\Rb^d$
without boundary, and let $f:\Rb^N \times \Rb^{d \times N} \to
[0,+\infty)$ be a Carath\'eodory function satisfying $(H_1)$ to
$(H_4)$. Then the family $\{\F_\e\}$ $\G$-converges for the strong
$L^1$-topology to the functional $\F_{\rm hom} : L^1(\O;\Rb^d) \to
[0,+\infty]$ defined by
$$\F_{\rm hom}(u):= \begin{cases} \ds
\begin{multlined}[9cm]
\,\int_\O Tf_{\rm hom}(u,\nabla u)\, dx +
\int_{\O\cap S_u}\vartheta_{\rm hom}(u^+,u^-,\nu_u)\, d\HH^{N-1}\,+ \\[-12pt]
+ \int_\O Tf^\infty_{\rm hom}\left(\tilde
u,\frac{dD^cu}{d|D^cu|}\right)\, d|D^cu|
\end{multlined}
& \text{if }u \in BV(\O;\M)\,,\\
& \\
\,+\infty & \text{otherwise}\,,
\end{cases}$$
where $Tf_{\rm hom}$ is given in (\ref{Tfhom}), $Tf_{\rm
hom}^\infty$ is the recession function of $Tf_{\rm hom}$ defined for
every $s \in \M$ and every $\xi \in [T_s(\M)]^N$ by
$$Tf_{\rm hom}^\infty(s,\xi):=\limsup_{t \to
+\infty}\,\frac{Tf_{\rm hom}(s,t\xi)}{t}\, ,$$
and for all $(a,b,\nu)
\in \M \times\M \times\SN$,
\begin{multline}\label{thetahom}
\vartheta_{\rm hom}(a,b,\nu) := \lim_{t\to+\infty}\inf_\varphi
\bigg\{\frac{1}{t^{N-1}} \int_{t\, Q_\nu} f^\infty(y,\nabla
\varphi(y))\, dy : \varphi \in W^{1,1}(tQ_\nu;\M)\,, \\
 \varphi=a \text{ on }\partial (tQ_\nu)\cap\{x\cdot \nu>0\}
\text{ and }\varphi=b \text{ on }\partial (tQ_\nu)\cap\{x\cdot \nu \leq 0\}\bigg\}\,,
\end{multline}
$Q_\nu$ being any open unit cube in $\Rb^N$ centered at the
origin with two of its faces orthogonal to $\nu$.
\end{theorem}

The paper is organized as follows. We first review in Section 2
standard facts about of manifold valued Sobolev mappings and
functions of bounded variation that will be used all the way
through. The main properties of the energy densities $Tf_{\rm hom}$ and
$\vartheta_{\rm hom}$ are  the object of
Section 3. A locality property of the $\Gamma$-limit is established in Section 4. The upper bound inequality
in  Theorem \ref{babmil2} is the object of Section~5.  The lower bound is obtained
 in Section 6 where  the proof of the theorem is completed.
Finally we state in the Appendix a relaxation
result for general manifolds and integrands which extends
\cite{AEL} and \cite{Mucci}.

\section{Preliminaries}

Let $\O$ be a generic
bounded open subset of $\Rb^N$. We write $\A(\O)$ for the family of
all open subsets of $\O$, and $\mathcal B(\O)$ for the
$\sigma$-algebra of all Borel subsets of $\O$. We also consider a
countable subfamily $\R(\O)$ of $\A(\O)$ made of all finite unions
of cubes with rational
edge length centered at rational points of $\Rb^N$.
Given $\nu \in \SN$, $Q_\nu$ stands for an open unit cube in $\Rb^N$
centered at the origin with two of its faces orthogonal to $\nu$ and
$Q_\nu(x_0,\rho):= x_0 + \rho \,Q_\nu$. Similarly $Q:=(-1/2,1/2)^N$
is the unit cube in $\Rb^N$ and $Q(x_0,\rho):= x_0 + \rho \,Q$. We
denote by $h^\infty$ the recession function of a generic scalar
function $h$, {\it i.e.},
$$h^\infty(\xi):=\limsup_{t\to+\infty}\,\frac{h(t\xi)}{t}\,.$$

The space of vector valued Radon measures in $\O$ with finite total
variation is denoted by $\M(\O;\Rb^m)$. We shall follow \cite{AFP}
for the standard notation on functions of bounded variation. We only
recall Alberti Rank One Theorem which states that for $|D^c u|$-a.e.
$x \in \O$, $$A(x):=\frac{dD^cu}{d|D^cu|}(x)$$ is a rank one matrix.

\bigskip

In this paper, we are interested in Sobolev and $BV$ maps taking
their values into a given manifold. We consider a connected smooth
submanifold $\M$ of $\Rb^d$ without boundary. The tangent space of
$\M$ at $s \in \M$ is denoted by $T_s(\M)$, ${\rm co}(\M)$ stands
for the convex hull of $\M$, and $\pi_1(\M)$ is the
fundamental group of $\M$.  

It is well known that if $u \in W^{1,1}(\O;\M)$, then  $\nabla u(x)
\in [T_{u(x)}(\M)]^N$ for $\LL^N$-a.e. $x \in \O$. The analogue
statement for $BV$-maps is given in Lemma  \ref{manifold} below.

\begin{lemma}\label{manifold}
For every $u \in BV(\O;\M)$,
\begin{align}
\label{aplimM}&\tilde u(x) \in \M\text{ for every } x \in \O\setminus S_u\,;\\[0.2cm]
\label{jumpM}&u^\pm(x) \in \M\text{ for every  }x \in J_u\,;\\[0.2cm]
\label{gradM}&\nabla u(x) \in [T_{u(x)}(\M)]^N \text{ for } \LL^N\text{-a.e. }x \in \O\,;\\[0.2cm]
\label{cantM}&\ds A(x):=\frac{dD^c u}{d|D^c u|}(x) \in [T_{\tilde u(x)}(\M)]^N \text{ for }|D^c u|\text{-a.e. }x \in \O\,.
\end{align}
\end{lemma}

\begin{proof}
We first show (\ref{aplimM}). By definition of the space
$BV(\O;\M)$, $u(y)\in\M$ for a.e. $y\in \O$. Therefore for any
$x\in\O\setminus S_u$, we have $|u(y)-\tilde u(x)|\geq
\text{dist}(\tilde u(x),\M)$ for a.e. $y\in\O$. By definition of
$S_u$, this yields $\text{dist}(\tilde u(x),\M)=0$, {\it i.e.},
$\tilde u(x)\in \M$. Arguing as for the approximate limit points,
one obtains (\ref{jumpM}).

Now it remains to prove \eqref{gradM} and \eqref{cantM}. We introduce the function $\Phi:\Rb^d\to\Rb$ defined by
$$\Phi(s)=\chi(\delta^{-1}\text{dist}(s,\M)^2)\, \text{dist}(s,\M)^2\,,$$
where $\chi\in \C_c^\infty(\Rb;[0,1])$ with $\chi(t)=1$ for $|t|\leq
1$, $\chi(t)=0$ for $|t|\geq2$, and $\delta>0$ is small enough so
that $\Phi \in \C^1(\Rb^d)$.  Note that for every $s \in \M$,
$\Phi(s)=0$ and
\begin{equation}\label{ker}
\text{Ker}\,\nabla\Phi(s)=T_s(\M)\,.
\end{equation}
By the Chain Rule formula in $BV$ (see, {\it e.g., }
\cite[Theorem~3.96]{AFP}), $\Phi\circ u\in BV(\O)$ and
\begin{align*}
D(\Phi\circ u)=&\,\nabla\Phi(u)\nabla u \,\mathcal{L}^N\res \, \O+\nabla \Phi(\tilde u) D^cu
+\big(\Phi(u^+)-\Phi(u^-)\big)\otimes\nu_u\,\mathcal{H}^{N-1}\res \, J_u\\
=&\,\nabla\Phi(u)\nabla u \,\mathcal{L}^N\res \, \O+\nabla \Phi(\tilde u)A|D^cu|\,,
\end{align*}
thanks to \eqref{jumpM}. On the other hand, $\Phi\circ u=0$ a.e. in
$\O$ since $u(x)\in\M$ for a.e. $x\in\O$. Therefore we have that
$D(\Phi\circ u)\equiv 0$. Since $\mathcal{L}^N\res \, \O$ and $|D^c
u|$ are mutually singular measures, we infer that
$\nabla\Phi(u(x))\nabla u(x)=0$ for $\mathcal{L}^N$-a.e.  $x\in\O$
and $\nabla \Phi(\tilde u(x))A(x)=0$ for $|D^cu|$-a.e. $x\in\O$.
Hence \eqref{gradM} and \eqref{cantM} follow from \eqref{ker}
together with \eqref{aplimM}.
\end{proof}

In \cite{B,BZ}, density results of smooth functions between
manifolds into Sobolev spaces have been established. In the
following theorem, we summarize these results only in $W^{1,1}$. Let
$\mathcal S$ be the family of all finite unions of subsets contained
in a $(N-2)$-dimensional submanifold of $\Rb^N$.

\begin{theorem}\label{density}Let $\D(\O;\M) \subset
W^{1,1}(\O;\M)$ be defined by
$$\D(\O;\M):=\begin{cases}
W^{1,1}(\O;\M)\cap\C^\infty(\O;\M) & \text{if $\,\pi_1(\M)=0$}\,,\\[10pt]
\big\{ u \in W^{1,1}(\O;\M)\cap  \C^\infty(\O \setminus \Sigma;\M)
\text{ for some } \Sigma \in \mathcal S \big\}
 & \text{otherwise}\,.
\end{cases}$$
Then $\D(\O;\M)$ is dense in $W^{1,1}(\O;\M)$ for the strong
$W^{1,1}(\O;\Rb^d)$-topology.
\end{theorem}

We now present a useful projection technique (taken from \cite{Dem}
for $\M=\Sd$). It was first introduced in \cite{HKL,HL}, and makes
use of an averaging device going back to \cite{FF}. We sketch the
proof for the convenience of the reader.

\begin{proposition}\label{proj}
Let $\M$ be a compact connected $m$-dimensional smooth submanifold
of $\Rb^d$ without boundary, and let $v \in W^{1,1}(\O;\Rb^d) \cap
\C^\infty(\O\setminus \Sigma;\Rb^d)$ for some $\Sigma\in\mathcal{S}$
such that $v(x) \in {\rm co}(\M)$ for a.e. $x \in \O$. Then there
exists $w \in W^{1,1}(\O;\M)$ satisfying $w=v$ a.e. in $\big\{x \in
\O\setminus\Sigma: v(x) \in \M \big\}$ and
\begin{equation}\label{1127}
\int_\O |\nabla w|\, dx \leq C_\star \int_\O |\nabla
v|\,dx\,,\end{equation} for some constant $C_\star >0$ which only
depends on $d$ and $\M$.
\end{proposition}

\begin{proof} According to \cite[Lemma 6.1]{HL} (which holds for
$p=1$), there exist a compact Lipschitz polyhedral set $X\subset\Rb^d$
of codimension greater or equal to $2$, and a locally Lipschitz map $\pi:\Rb^d \setminus X \to
\M$ such that
\begin{equation}\label{gradproj}
\int_{B^d(0,R)} |\nabla \pi(s)|\, ds <+\infty \quad \text{ for every
}R<+\infty\,.
\end{equation}
Moreover, in a neighborhood of $\M$ the mapping $\pi$ is smooth of
constant rank  equal to $m$.

We argue as in the proof of \cite[Theorem 6.2]{HL}. Let $B$ be an
open ball in $\Rb^d$ containing $\M \cup X$, and let $\d>0$ small
enough so that the nearest point projection on $\M$ is a well
defined smooth mapping in the $\d$-neighborhood of $\M$. Fix $\sigma <
\inf\{\d,\dist({\rm co}(\M),\partial B)\}$ small enough, and for $a
\in B^d(0,\sigma)$ we define the translates
$B_a:=a+B$ and $X_a:=a+X$,
and the projection $\pi_a:B_a \setminus X_a \to \M$ by
$\pi_a(s):=\pi(s-a)$. Since $\pi$ has full rank and is smooth in a neighborhood of $\M$,
by the Inverse Function Theorem the number
\begin{equation}\label{gradlambda}
\Lambda:=\sup_{a \in B^d(0,\sigma)} {\rm
Lip}\big({\pi_a}_{|\M}\big)^{-1}
\end{equation}
is finite and only depends on $\M$. Using Sard's lemma, one can show
that  $\pi_a \circ v \in W^{1,1}(\O;\M)$ for
$\LL^d$-a.e. $a \in B^d(0,\sigma)$. Then Fubini's theorem together
with the Chain Rule formula yields
\begin{multline*}
\int_{B^d(0,\sigma)}
\int_{\O} |\nabla (\pi_a \circ v)(x)|\, d\LL^N(x) \, d\LL^d(a)
\leq\,\\
\leq \int_\O |\nabla v(x)|
\left(\int_{B^d(0,\sigma)}|\nabla \pi(v(x)-a)| \, d\LL^d(a) \right)\, d\LL^N(x)
 \leq \,\\
 \leq  \left(\int_{B} |\nabla \pi(s)|\, d\LL^d(s)\right)
\left(\int_\O |\nabla v(x)|\,d\LL^N(x)\right)\,.
\end{multline*}
Therefore we can find $a \in B^d(0,\sigma)$ such that
\begin{equation}\label{gradpia}\int_\O |\nabla (\pi_a \circ v)|\, dx
\leq C\LL^d\left(B^d(0,\sigma)\right)^{-1} \int_\O |\nabla
v|\,dx\,,\end{equation} where we used (\ref{gradproj}). To conclude,
it suffices to set $w:=\big({\pi_a}_{|\M}\big)^{-1} \circ \pi_a \circ v$,
and (\ref{1127}) arises as a consequence of (\ref{gradlambda}) and
(\ref{gradpia}).
\end{proof}

\section{Properties of homogenized energy densities}

In this section we present the main properties of the energy
densities $Tf_\hom$ and $\vartheta_\hom$ defined in (\ref{Tfhom})
and (\ref{thetahom}). In particular we will prove that
$\vartheta_\hom$ is well defined in the sense that the limit in
(\ref{thetahom}) exists.

\subsection{The tangentially homogenized bulk energy}\label{thbe}

We start by considering the bulk energy density $Tf_{\rm hom}$ defined in \eqref{Tfhom}. 
As in \cite{BM} we first construct a new energy density
$g:\Rb^N\times\Rb^d\times\Rb^{d\times N}\to[0,+\infty)$ satisfying
$$g(\cdot,s,\xi)=f(\cdot,\xi)\quad \text{and}\quad g_{\rm hom}(s,\xi)=Tf_{\rm
hom}(s,\xi)\quad \text{for $s\in\M$ and $\xi\in[T_s(\M)]^N$}\,.$$
Hence upon extending $Tf_\hom$ by $g_\hom$ outside the set
$\big\{(s,\xi) \in \Rb^d \times \Rb^{d \times N}:\; s \in \M,\, \xi
\in [T_s(\M)]^N\big\}$, we will tacitly assume $Tf_\hom$ to be
defined over the whole $\Rb^d \times \Rb^{d \times N}$. We proceed
as follow. \vskip5pt

For $s \in \M$ we denote by $P_s:\Rb^d \to T_s(\M)$ the orthogonal
projection from $\Rb^d$ into $T_s(\M)$, and we set
$$\mathbf{P}_s(\xi):=(P_s(\xi_1),\ldots,P_s(\xi_N)) \quad \text{for
$\xi=(\xi_1,\ldots,\xi_N)\in\Rb^{d\times N}\,$.}$$ For $\delta_0>0$
fixed, let  $\mathcal U:=\big\{s
\in\Rb^d\,:\,\dist(s,\M)<\d_0\big\}$ be the $\d_0$-neighborhood of
$\M$. Choosing $\delta_0>0$ small enough, we may assume that the
nearest point projection $\Pi: \mathcal U \to \M$ is a well defined
Lipschitz mapping. Then the map $s \in \mathcal U \mapsto
P_{\Pi(s)}$ is Lipschitz. Now we introduce a cut-off function $\chi
\in \C^\infty_c(\Rb^d;[0,1])$ such that $\chi(t)=1$ if $\dist(s,\M)
\leq \delta_0/2$, and $\chi(s)=0$ if $\dist(s,\M) \geq 3\delta_0/4$,
and we define
$$\mathbb{P}_s(\xi):=\chi(s) \mathbf{P}_{\Pi(s)}(\xi)\quad \text{for $(s,\xi) \in \Rb^d \times \Rb^{d \times N}\,$.}$$
Given the Carath\'eodory integrand $f:\Rb^N \times \Rb^{d\times N}
\to [0,+\infty)$ satisfying assumptions $(H_1)$ to $(H_3)$, we
construct the new integrand $g: \Rb^N\times\Rb^d\times\Rb^{d\times
N}\to[0,+\infty)$ as
\begin{equation}\label{defig}
g(y,s,\xi):=f(y,\mathbb{P}_s(\xi))+|\xi-\mathbb{P}_s(\xi)|\,.
\end{equation}

We summarize in the following lemma the main properties of $g$.

\begin{lemma}\label{defg}
The integrand $g$ as defined in (\ref{defig})  is a Carath\'eodory function satisfying
\begin{equation}\label{idfg}
g(y,s,\xi)=f(y,\xi)\quad\text{and}\quad g^\infty(y,s,\xi)=f^\infty(y,\xi)\quad \text{for $s
\in \M$ and $\xi \in [T_s(\M)]^N\,$,}
\end{equation}
and
\begin{itemize}
\item[(i)] $g$ is $1$-periodic in the first variable;
\item[(ii)] there exist $0<\alpha'\leq \b'$ such that
\begin{equation}\label{pgrowth}
\alpha'|\xi|\leq g(y,s,\xi)\leq
\b'(1+|\xi|)\quad\text{for every $(s,\xi)\in \Rb^d \times \Rb^{d
\times N}$ and a.e. $y \in \Rb^N$}\,;\end{equation}
\item[(iii)] there exist $C>0$ and $C'>0$ such that
\begin{equation}\label{moduluscont}
|g(y,s,\xi)-g(y,s',\xi)| \leq C|s-s'|\; |\xi|\,,\end{equation}
\begin{equation}\label{lipg}
|g(y,s,\xi) - g(y,s,\xi')| \leq C'|\xi-\xi'|\end{equation} for every
$s$, $s' \in \Rb^d$, every $\xi \in \Rb^{d \times N}$ and a.e. $y
\in \Rb^N$;
\item[(iv)] if in addition $(H_4)$ holds, there  exists $0<q<1$ and $C''>0$ such that
\begin{equation}\label{grec}
|g(y,s,\xi) - g^\infty(y,s,\xi)| \leq C''(1+|\xi|^{1-q}) \end{equation}
for every $(s,\xi) \in \Rb^d \times \Rb^{d \times N}$  and a.e.
$y \in \Rb^N\,$.
\end{itemize}
\end{lemma}
\vskip5pt

We can now state the properties of $Tf_\hom$ and the relation
between $Tf_\hom$ and $g_\hom$ through the homogenization procedure.

\begin{proposition}\label{properties1}
Let $f:\Rb^N \times \Rb^{d\times N} \to [0,+\infty)$ be a
Carath\'eodory integrand satisfying $(H_1)$ to $(H_3)$.
Then the following properties hold:
\begin{itemize}
\item[(i)] for every $s \in \M$ and $\xi \in [T_s(\M)]^N$,
\begin{equation}\label{identhomform}
Tf_{\rm hom}(s,\xi)=g_{\rm hom}(s,\xi)\,,
\end{equation}
where
$$g_{\rm hom}(s,\xi):=\lim_{t\to+\infty}\inf_{\varphi}\bigg\{
\med_{(0,t)^N} g(y,s,\xi+\nabla \varphi(y))\, dy: \varphi
\in W^{1,\infty}_0((0,t)^N;\Rb^d) \bigg\}$$ is the usual homogenized
energy density of $g$ (see, e.g.,
\cite[Chapter~14]{BD});\\
\item[(ii)] the function $Tf_{\rm hom}$ is tangentially
quasiconvex, {\it i.e.}, for all $s \in \M$ and all $\xi \in
[T_s(\M)]^N$,
$$Tf_{\rm hom}(s,\xi) \leq \int_Q Tf_{\rm hom}(s,\xi + \nabla \varphi(y))\, dy$$
for every $\varphi \in W^{1,\infty}_0(Q;T_s(\M))$. In particular
$Tf_{\rm hom}(s,\cdot)$ is rank one convex;\\
\item[(iii)] there exists
$C>0$ such that
\begin{equation}\label{pgT}
\a|\xi|\leq Tf_{\rm hom}(s,\xi) \leq \b(1+|\xi|)\,,
\end{equation}
and
\begin{equation}\label{plipT}
|Tf_{\rm hom}(s,\xi)-Tf_{\rm hom}(s,\xi')| \leq
C|\xi-\xi'|
\end{equation}
for every $s \in \M$ and $\xi$, $\xi' \in [T_s(\M)]^N$;
\item[(iv)] there exists $C_1>0$ such that
\begin{equation}\label{hyp4}
|Tf_{\rm hom}(s,\xi)-Tf_{\rm hom}(s',\xi)| \leq
C_1|s-s'|(1+|\xi|)\,,
\end{equation}
for every $s$, $s' \in \Rb^d$ and $\xi \in \Rb^{d \times N}$. In
particular $Tf_{\rm hom}$ is continuous;\\
\item[(v)] if in addition $(H_4)$ holds, there exist $C_2>0$ and $0<q<1$ such that
\begin{equation}\label{hyp5}
|Tf^{\,\infty}_{\rm hom}(s,\xi)-Tf_{\rm hom}(s,\xi)| \leq
C_2(1+|\xi|^{1-q})\,,
\end{equation}
for every $(s,\xi)\in \Rb^d\times \Rb^{d \times N}$.
\end{itemize}
\end{proposition}

\begin{remark}\label{reminfty}
Observe that, if $f$ satisfies assumption $(H_3)$, then $f^\infty$
satisfies $(H_3)$ as well. In particular the function $f^\infty$ is
Carath\'eodory, $1$-periodic in the first variable, and positively
1-homogeneous with respect to the second variable. In view of the
growth and coercivity condition $(H_2)$, one gets that
\begin{equation}\label{finfty1gc} \a |\xi| \leq f^\infty(y,\xi) \leq
\b|\xi| \quad \text{ for all }\xi \in \Rb^{d \times N} \text{ and
a.e. }y \in \Rb^N\,.
\end{equation}
Then, as for $f^\infty$, the function $g^\infty$ is Carath\'eodory, $1$-periodic in the first variable, and positively 1-homogeneous with respect to the second variable. Moreover,
$$\alpha'|\xi|\leq g^\infty(y,s,\xi)\leq
\b'|\xi|\quad\text{for every $(s,\xi)\in \Rb^d \times \Rb^{d \times
N}$ and a.e. $y \in \Rb^N$}\,, $$ and $g^\infty$ satisfies estimates
analogue to \eqref{moduluscont} and \eqref{lipg}. Hence we may apply
classical homogenization results to $g^\infty$. In addition, in view
of \eqref{idfg}, claim{\it (i)} in Proposition \ref{properties1}
holds for $f^\infty$ and $g^\infty$, and we have
$$T(f^\infty)_{\rm hom}(s,\xi)=(g^\infty)_{\rm hom}(s,\xi) \quad \text{for every $s
\in \M$ and $\xi \in [T_s(\M)]^N\,$.}$$ In particular
$T(f^\infty)_{\rm hom}$ will be tacitely extended by
$(g^\infty)_{\rm hom}$.
\end{remark}

\noindent {\bf Proof of Proposition \ref{properties1}.} The proofs
of claims {\it (i)-(iii)} can be obtained exactly as in
\cite[Proposition 2.1]{BM} and we shall omit it. It remains to prove
{\it (iv)} and {\it (v)}.

Fix $s,s'\in\Rb^d$ and $\xi \in \Rb^{d \times N}$. For any $\eta>0$,
we may find $k \in \Nb$ and $\varphi \in
W_0^{1,\infty}((0,k)^N;\Rb^d)$ such that
$$\med_{(0,k)^N} g(y,s,\xi+\nabla \varphi)\, dy \leq g_\hom(s,\xi)
+ \eta\,.$$ We infer from \eqref{pgrowth}
that $\alpha'|\xi|\leq g_{\rm hom}(s,\xi)\leq \beta'(1+|\xi|)$ and consequently
\begin{equation}\label{varphi}
\med_{(0,k)^N}|\nabla \varphi|\, dy \leq C(1+|\xi|)\,,
\end{equation}
for some constant $C>0$ depending only on $\a'$ and $\b'$.
Then from (\ref{identhomform}) and \eqref{moduluscont} it follows
that
\begin{multline*}
Tf_\hom(s',\xi)-Tf_\hom(s,\xi)= g_\hom(s',\xi)-g_\hom(s,\xi)\leq\\
\leq \med_{(0,k)^N} \big(g(y,s',\xi+\nabla \varphi)-g(y,s,\xi+\nabla
\varphi)\big)\, dy+ \eta \leq\\ \leq  C |s-s'|\med_{(0,k)^N}|\xi
+\nabla \varphi|\, dy+\eta\leq  C|s-s'|(1+|\xi|)+\eta\,.
\end{multline*}
We deduce relation (\ref{hyp4}) inverting the roles of
$s$ and $s'$, and sending $\eta$ to zero. In particular, we obtain
that $Tf_\hom$ is continuous as a consequence of (\ref{hyp4}) and
(\ref{plipT}).

To show (\ref{hyp5}), let us consider sequences $t_n \nearrow +\infty$, $k_n
\in \Nb$ and $\varphi_n\in W^{1,\infty}_0((0,k_n)^N;T_s(\M))$ such
that
\begin{equation}\label{comptfhoninf}
Tf_\hom^{\, \infty}(s,\xi)=\lim_{n \to
+\infty}\frac{Tf_\hom(s,t_n\xi)}{t_n}\,,
\end{equation}
and
$$\med_{(0,k_n)^N}f(y,t_n \xi + t_n \nabla \varphi_n)\, dy \leq
Tf_\hom(s,t_n\xi)+\frac{1}{n}\,.$$ 
Then $(H_2)$ and
(\ref{pgT}) yield
\begin{equation}\label{varphi_n}\med_{(0,k_n)^N}|\nabla \varphi_n|\, dy \leq
C(1+|\xi|)\,,
\end{equation}
for some constant $C>0$ depending only on
$\a$ and $\b$. Using $(H_4)$ and \eqref{comptfhoninf}, we derive that
\begin{align*}
 Tf_\hom(s,\xi) - Tf_\hom^{\,\infty}(s,\xi)
\leq & \liminf_{n \to +\infty} \Bigg\{\med_{(0,k_n)^N}\bigg|
f(y,\xi+\nabla\varphi_n)- f^{\,\infty}(y,\xi +\nabla \varphi_n)\bigg|\, dy\, +\\
&\,+ \med_{(0,k_n)^N}\bigg| f^{\,\infty}(y,\xi+\nabla
\varphi_n)-\frac{ f(y,t_n \xi + t_n \nabla
\varphi_n)}{t_n}\bigg|\, dy\Bigg\}\\
\leq & \liminf_{n \to +\infty} \Bigg\{C
\med_{(0,k_n)^N}\!(1+|\xi+\nabla \varphi_n|^{1-q})\, dy\\
&+ \frac{C}{t_n}\med_{(0,k_n)^N}\!(1+t_n^{1-q}|\xi+\nabla
\varphi_n|^{1-q})\, dy\Bigg\}\,,
\end{align*}
where we have also used the fact that $f^{\,\infty}(y,\cdot)$ is
positively homogeneous of degree one in the last inequality. Then
(\ref{varphi_n}) and H\"older's inequality lead to
\begin{equation}\label{firstineq}
Tf_\hom(s,\xi) - Tf_\hom^{\,\infty}(s,\xi) \leq C(1+|\xi|^{1-q})\,.
\end{equation}
Conversely, given $k\in \Nb$ and $\varphi \in
W_0^{1,\infty}((0,k)^N;T_s(\M))$, we deduce from $(H_2)$
that
$$\frac{f(\cdot,t(\xi+\nabla\varphi(\cdot)))}{t} \leq \b(1+|\xi+\nabla
\varphi|) \in L^1((0,k)^N)$$ whenever $t > 1$. Then Fatou's lemma
implies
\begin{equation*}
Tf_\hom^\infty(s,\xi)
\leq \limsup_{t \to +\infty} \med_{(0,k)^N}\frac{f(y,t\xi+t
\nabla \varphi)}{t}\, dy \leq  \med_{(0,k)^N}f^\infty(y,\xi+ \nabla \varphi)\, dy\,.
\end{equation*}
Taking the infimum over all admissible $\varphi$'s  and letting $k\to+\infty$, we
infer
\begin{equation}\label{remdensinf}
Tf_\hom^\infty(s,\xi) \leq T(f^\infty)_\hom(s,\xi)\,.
\end{equation}
For $\eta>0$ arbitrary small, consider $k \in \Nb$ and $\varphi \in
W^{1,\infty}_0((0,k)^N;T_s(\M))$ such that
$$\med_{(0,k)^N}f(y,\xi + \nabla \varphi)\, dy \leq
Tf_\hom(s,\xi)+\eta\,.$$
In view of $(H_2)$ and (\ref{pgT}), it turns
out that \eqref{varphi} holds 
with constant $C>0$ only depending  on $\a$ and $\b$. Then it
follows from \eqref{remdensinf} that
\begin{multline*}
Tf_\hom^\infty(s,\xi) - Tf_\hom(s,\xi)  \leq
T(f^\infty)_\hom(s,\xi) - Tf_\hom(s,\xi)\leq\\
\leq  \med_{(0,k)^N}|f^\infty(y,\xi + \nabla \varphi)- f(y,\xi
+ \nabla \varphi)|\, dy+\eta \leq  C\med_{(0,k)^N}(1+|\xi + \nabla \varphi|^{1-q})\, dy+\eta\,,
\end{multline*}
where we have used $(H_4)$ in the last inequality. Using H\"older's
inequality, relation (\ref{varphi}) together with the arbitrariness of
$\eta$ yields
\begin{equation}\label{secineq}Tf_\hom^\infty(s,\xi) -
Tf_\hom(s,\xi) \leq C(1+|\xi|^{1-q})\,.\end{equation} Gathering
(\ref{firstineq}) and (\ref{secineq}) we conclude the proof of
(\ref{hyp5}). \prbox

\subsection{The homogenized surface energy}\label{sectsurf}

\noindent We now present the homogenized surface energy density
$\vartheta_\hom$. We start by introducing some useful notations.

Given $\nu=(\nu_1,\ldots,\nu_N)$ an orthonormal basis of $\Rb^N$ and
$(a,b) \in \M \times \M $, we denote by
$$Q_\nu:=\Big\{\alpha_1\nu_1+\ldots+\alpha_N\nu_N\;:\;\alpha_1,\ldots,\alpha_N \in(-1/2,1/2)\Big\}\,,$$
and for $x \in \Rb^N$, we set $\|x\|_{\nu,\infty}:=\sup_{i \in
\{1,\ldots,N\}}|x\cdot\nu_i|$, $x_\nu:=x\cdot \nu_1$ and
$x':=(x\cdot\nu_2)\nu_2 +\ldots+(x\cdot \nu_N)\nu_N$ so that $x$ can
be identified to the pair $(x',x_\nu)$. Let $u_{a,b,\nu}:Q_\nu \to
\M$ be the function defined by
$$u_{a,b,\nu}(x):=\begin{cases}
a & \text{if }  x_\nu>0\,,\\[5pt]
b & \text{if } x_\nu \leq 0\,.
\end{cases}$$
We introduce the class of functions
$$\A_t(a,b,\nu) :=\Big\{\varphi \in W^{1,1}(t Q_\nu;\M):
\varphi=u_{a,b,\nu} \text{ on }\partial(tQ_\nu)\Big\}\,.$$
We have the following result.

\begin{proposition}\label{limitsurfenerg}
For every $(a,b,\nu_1)\in\M\times\M\times\SN$, there exists
\begin{eqnarray*}
\vartheta_{\rm hom}(a,b,\nu_1) &: = & \lim_{t\to+\infty}\,\inf_\varphi
\left\{\frac{1}{t^{N-1}} \int_{t Q_\nu} f^\infty(y,\nabla
\varphi(y))\, dy : \varphi \in \A_t(a,b,\nu) \right\}\,,
\end{eqnarray*}
where $\nu=(\nu_1,\ldots,\nu_N)$ is any orthonormal basis of $\Rb^N$ with first element equal to $\nu_1$ (the limit being independent of such a
choice).
\end{proposition}

The proof of Proposition \ref{limitsurfenerg} is quite indirect and
is based on an analogous result for a similar surface energy density
$\tilde \vartheta_\hom$ (see \eqref{surfen2} below). We will prove in Proposition \ref{limsurf2}
that the two densities coincide.
\vskip5pt

Given $a$ and $b\in \M$, we introduce the family of geodesic curves
 between $a$ and $b$ by
$$\mathcal{G}(a,b):=\bigg\{\g\in \C^{\infty}(\Rb;\M):\;\g(t)=a\text{ if } t\geq 1/2,\, \g(t)=b\text{ if }
t\leq -1/2\,,\;\int_\Rb|\dot \g|\, dt=\mathbf
d_\M(a,b)\bigg\}\,,$$ where $\mathbf d_\M$ denotes the geodesic
distance on $\M$. We define for $\e>0$ and
$\nu=(\nu_1,\ldots,\nu_N)$ an orthonormal basis of $\Rb^N$,
$$\mathcal{B}_\e(a,b,\nu):=\Big\{u \in W^{1,1}(Q_\nu;\M)\;:\;u(x)=\g(x_\nu/\e)
\text{ on }\partial Q_\nu\text{ for some }
\g\in\mathcal{G}(a,b)\Big\}\,.$$

\begin{proposition}\label{limsurf2}
For every $(a,b)\in\M\times\M$ and every orthonormal basis
$\nu=(\nu_1,\ldots,\nu_N)$ of $\Rb^N$, there exists the limit
\begin{equation}\label{surfen2}
\tilde\vartheta_{\rm hom}(a,b,\nu)  :=  \lim_{\e\to 0}\,\inf_u \left\{
\int_{Q_\nu} f^\infty\left(\frac{x}{\e},\nabla u\right) dx : u \in
\mathcal{B}_\e(a,b,\nu) \right\}\,.
\end{equation}
Moreover $\tilde\vartheta_{\rm hom}(a,b,\nu)$ only depends on $a$, $b$ and $\nu_1$.
\end{proposition}

\begin{proof}
The proof follows the scheme of the one in
\cite[Proposition~2.2]{BDV}. We fix $a$ and $b\in \M$. For every
$\e>0$ and every orthonormal basis $\nu=(\nu_1,\ldots,\nu_N)$ of
$\Rb^N$, we set
$$ I_\e(\nu)= I_\e(a,b,\nu):= \inf \left\{ \int_{Q_\nu}
f^\infty\left(\frac{x}{\e},\nabla u\right) dx : u \in
\mathcal{B}_\e(a,b,\nu) \right\}\,.
$$
We divide the proof into several steps.\vskip5pt

{\bf Step 1.} Let $\nu$ and $\nu'$ be two orthonormal bases of
$\Rb^N$ with equal first vector, {\it i.e.}, $\nu_1=\nu'_1$. Suppose
that $\nu$ is a rational basis, {\it i.e.}, for all
$i\in\{1,\ldots,N\}$ there exists $\gamma_i\in\Rb\setminus\{0\}$
such that $v_i:=\gamma_i\nu_i\in\Zb^N$. Similarly to Step 1 of the
proof of \cite[Proposition 2.2]{BDV}, we readily obtain that
\begin{equation}\label{complimsupinf}
\limsup_{\e \to0}\, I_\e (\nu')\leq \liminf_{\e \to0}\, I_\e(\nu)\,.
\end{equation}

{\bf Step 2.} Let $\nu$ and $\nu'$ be two orthonormal rational bases
of $\Rb^N$ with equal first vector. By Step~1 we immediately obtain
that the limits $\ds \lim_{\e\to0}I_\e(\nu)$ and
$\ds\lim_{\e\to0}I_\e(\nu')$ exist and are equal.

\vskip5pt

{\bf Step 3.}  We claim that for every $\sigma>0$ there exists
$\delta>0$ (independent of $a$ and $b$) such that if $\nu$ and
$\nu'$ are two orthonormal bases of $\Rb^N$ with
$|\nu_i-\nu'_i|<\delta$ for every $i=1,\ldots,N$, then
$$\liminf_{\e\to0} I_\e(\nu)-K\sigma\leq \liminf_{\e\to0}I_\e(\nu')\leq\limsup_{\e\to0}I_\e(\nu')\leq
\limsup_{\e\to0} I_\e(\nu)+K\sigma$$ where $K$ is a positive
constant which only depends on $\M$, $\beta$ and $N$.

We use the notation $Q_{\nu,\eta}:=(1-\eta)Q_\nu$ where $0<\eta<1$.
Let $\sigma>0$ be fixed and let $0<\eta<1$ be such that
\begin{equation}\label{condeta}
\eta<\frac{1}{34}\quad\text{and}\quad\max\bigg\{1-(1-\eta)^{N-1}\,,\;
\frac{(1-\eta)^{N-1}(1-2\eta)^{N-1}}{(1-3\eta)^{N-1}}-(1-2\eta)^{N-1}\bigg\}<\sigma.
\end{equation}
Consider $\delta_0>0$ (that may be chosen so that $\d_0 \leq
\eta/(2\sqrt{N})$) such that for every $0<\delta\leq \delta_0$ and
every pair $\nu$ and $\nu'$ of orthonormal basis of $\Rb^N$
satisfying $|\nu_i-\nu'_i|\leq \delta$ for $i=1,\ldots,N$, one has
\begin{equation}\label{approxnu}
Q_{\nu,3\eta}\subset Q_{\nu',2\eta}\subset Q_{\nu,\eta}\,,
\end{equation}
and $\{x\cdot\nu_1'=0\}\cap \partial Q_{\nu,\eta} \subset \{|x\cdot\nu_1|\leq 1/8\}$.

Given $\e>0$ small, we consider $u_\e\in\mathcal{B}_{\e}(a,b,\nu')$
such that
$$\int_{Q_{\nu'}}f^\infty\left(\frac{x}{\e},\nabla u_\e\right)dx\leq I_{\e}(\nu')
+\sigma\,,$$
 where $u_\e(x)=\g_\e(x_{\nu'}/\e)$ for $x\in\partial
Q_{\nu'}$. Now we construct
$v_\e\in\mathcal{B}_{(1-2\eta)\e}(a,b,\nu)$ satisfying the boundary condition
$v_\e(x)=\g_\e\big(x_\nu/(1-2\eta)\e\big)$ for $x\in\partial
Q_{\nu}$. Consider $F_\eta:\Rb^N \to \Rb$,
$$F_\eta(x):=\bigg(\frac{1- 2\|x'\|_{\nu,\infty}}{\eta}\bigg)\frac{x_{\nu'}}{1-2\eta}+
\bigg(\frac{\eta-1+
2\|x'\|_{\nu,\infty}}{\eta}\bigg)\frac{x_\nu}{1-2\eta}\,,$$
and define
$$v_\e(x):=\begin{cases}
\ds u_\e\bigg(\frac{x}{1-2\eta}\bigg) & \text{if }x\in
Q_{\nu',2\eta},\\[10pt]
\ds \g_\e\bigg(\frac{x_{\nu'}}{(1-2\eta)\e}\bigg) &\text{if }
x\in Q_{\nu,\eta}\setminus Q_{\nu',2\eta}\,,\\[8pt]
\ds a & \ds\text{if } x \in Q_\nu \setminus Q_{\nu,\eta} \text{ and }x_\nu \geq \frac{1}{4}\,,\\[8pt]
\ds \gamma_\e\bigg(\frac{F_\eta(x)}{\e}\bigg) & \text{if }x \in
A_\eta :=\big\{x: |x_\nu|\leq 1/4\big\}\cap(Q_\nu\setminus
Q_{\nu,\eta})\,,\\[8pt]
\ds b & \ds\text{if } x \in Q_\nu \setminus Q_{\nu,\eta} \text{ and
}x_\nu \leq -\frac{1}{4}\,.
\end{cases}$$
We can check that $v_\e$ is well defined for $\e$ small enough
and that $v_\e\in\mathcal{B}_{(1-2\eta)\e}(a,b,\nu)$. Therefore
\begin{align}\label{Ebis}
I_{(1-2\eta)\e}(\nu)\leq &
\int_{Q_\nu}f^\infty\bigg(\frac{x}{(1-2\eta)\e},\nabla v_\e\bigg)\,
dx\nonumber\\
=&\int_{Q_{\nu',2\eta}} f^\infty\bigg(\frac{x}{(1-2\eta)\e},\nabla
v_\e\bigg)\,dx+\int_{Q_{\nu,\eta}\setminus
Q_{\nu',2\eta}}f^\infty\bigg(\frac{x}{(1-2\eta)\e},\nabla
v_\e\bigg)\, dx\nonumber\\
&+\int_{A_{\eta}}f^\infty\bigg(\frac{x}{(1-2\eta)\e},\nabla
v_\e\bigg)\, dx=:\,I_1+I_2+I_3\,.
\end{align}
We now estimate these three integrals. First, we easily get that
\begin{equation}\label{E1bis}I_1= (1-2\eta)^{N-1}\int_{Q_{\nu'}}f^\infty\left(\frac{y}{\e},\nabla u_\e\right)
dy\leq I_\e(\nu')+\sigma\,. \end{equation} In view of
\eqref{approxnu} we have $Q_{\nu,\eta}\subset
(1-\eta)(1-2\eta)(1-3\eta)^{-1}Q_{\nu'}=:D_{\eta}$. Then we infer
from the growth condition (\ref{finfty1gc}) together with Fubini's
theorem that
\begin{eqnarray}\label{E2bis}
I_2 & \leq & \beta\int_{D_\eta\setminus Q_{\nu',2\eta}}|\nabla
v_\e|\, dx = \frac{\beta}{(1-2\eta)\e} \int_{(D_\eta\setminus
Q_{\nu',2\eta})\cap\{|x_{\nu'}|\leq (1-2\eta)\e/2\}}
\bigg|\, \dot \g_\e\bigg(\frac{x_{\nu'}}{(1-2\eta)\e}\bigg)\bigg|\, dx\nonumber\\
&= & \beta\mathcal{H}^{N-1}\big((D_\eta\setminus Q_{\nu',2\eta})\cap
\{x_{\nu'}=0\}\big)\frac{1}{(1-2\eta)\e}\int_{-(1-2\eta)\e/2}^{(1-2\eta)\e/2}
\bigg|\, \dot \g_\e\bigg(\frac{t}{(1-2\eta)\e}\bigg)\bigg|\, dt\nonumber\\
&=& \beta \mathbf
d_\M(a,b)\bigg(\frac{(1-\eta)^{N-1}(1-2\eta)^{N-1}}{(1-3\eta)^{N-1}}-(1-2\eta)^{N-1}\bigg)\,.
\end{eqnarray}
Now it remains to estimate $I_3$. To this purpose we first observe
that \eqref{approxnu} yields \begin{equation}\label{Feta1} \|\nabla
F_\eta\|_{L^\infty(A_\eta;\Rb^N)}\leq C\,,
\end{equation}
for some absolute
constant $C>0$, and
\begin{equation}\label{Feta2}|\nabla
F_\eta(x)\cdot \nu_1|\geq 1\quad\text{for a.e. $x\in
A_\eta\,$.}
\end{equation}
Hence, thanks the growth condition
(\ref{finfty1gc}), (\ref{Feta1}) and (\ref{Feta2}), we get that
\begin{multline*}
I_3  \leq  \beta\int_{A_\eta}|\nabla v_\e|\, dx \leq
\frac{C\beta}{\e} \int_{A_\eta}\bigg|\, \dot
\g_\e\bigg(\frac{F_\eta(x)}{\e}\bigg)\bigg |\, dx \leq  \frac{C \beta}{\e} \int_{A_\eta}\bigg|\, \dot
\g_\e\bigg(\frac{F_\eta(x)}{\e}\bigg)\bigg |\, |\nabla
F_\eta(x)\cdot\nu_1| \,dx=\\
=C\beta\int_{A'_\eta}\bigg(\frac{1}{\e}\int_{-1/4}^{1/4}\bigg|\,
\dot \g_\e\bigg(\frac{F_\eta(t\nu_1+x')}{\e}\bigg)\bigg |\, |\nabla
F_\eta(t\nu_1+x')\cdot\nu_1|\, dt\bigg)\, d\mathcal{H}^{N-1}(x')\,,
\end{multline*}
where we have set $A'_\eta:=A_\eta\cap\{x_\nu=0\}$, and used Fubini's
theorem in the last equality. Changing variables
$s=(1/\e)F_{\eta}(t\nu_1+x')$, we obtain that for
$\mathcal{H}^{N-1}$-a.e. $x' \in A'_\eta$,
$$\frac{1}{\e}\int_{-1/4}^{1/4}\bigg|\, \dot \g_\e\bigg(\frac{F_\eta(t\nu_1+x')}{\e}\bigg)\bigg|\,
|\nabla F_\eta(t\nu_1+x')\cdot\nu_1|\, dt\leq \int_\Rb |\dot
\g_\e(s)|\,ds=\mathbf d_\M(a,b)\,.$$ Consequently,
\begin{equation}\label{E3bis}I_3\leq C\beta\, \mathcal{H}^{N-1}(A'_\eta)\, \mathbf
d_\M(a,b)=C\beta \big(1-(1-\eta)^{N-1})\, \mathbf
d_\M(a,b)\,.\end{equation} In view of (\ref{Ebis}), (\ref{condeta})
and estimates  (\ref{E1bis}), (\ref{E2bis}) and (\ref{E3bis}),  we
conclude that
$$I_{(1-2\eta)\e}(\nu)\leq I_\e(\nu')+K\sigma\,,$$
where $K=1+\beta \Delta(1+C)$, $\Delta$ is the diameter of $\M$ and
$C$ is the constant given by (\ref{Feta1}). Finally, letting
$\e\to0$ we derive
$$\liminf_{\e\to 0}I_\e(\nu)\leq \liminf_{\e\to0}I_\e(\nu')+K\sigma\,, \text{ and }
 \limsup_{\e\to 0}I_\e(\nu)\leq \limsup_{\e\to0}I_\e(\nu')+K\sigma\,. $$
The symmetry of the roles of $\nu$ and $\nu'$ allows us to invert
them, thus concluding the proof of Step~3.\vskip5pt

{\bf Step 4.} Let $\nu$ and $\nu'$ be two orthonormal bases of
$\Rb^N$ with equal first vector. Similarly to Step~4 of the proof of
\cite[Proposition 2.2]{BDV}, by Steps 2 and 3 we readily obtain that
the limits $\ds\lim_{\e\to0}I_\e(\nu)$ and
$\ds\lim_{\e\to0}I_\e(\nu')$ exist and are equal.
\end{proof}

\noindent{\bf Proof of Proposition \ref{limitsurfenerg}.}
We use the notation of the previous proof. Given $\e>0$ and an
orthonormal basis $\nu=(\nu_1,\ldots,\nu_N)$ of $\Rb^N$, we set
\begin{align*}
J_\e(\nu)=J_\e(a,b,\nu) : = & \inf \left\{ \int_{Q_\nu}
f^\infty\left(\frac{x}{\e},\nabla u\right)\, dx : u \in
\mathcal{A}_1(a,b,\nu) \right\} \\
=&\inf\bigg\{\e^{N-1}\int_{\frac{1}{\e}Q_\nu}f^\infty(y,\nabla
\varphi)\,dy\;:\;\varphi\in \mathcal{A}_{1/\e}(a,b,\nu)
\bigg\}\,.\end{align*}
We claim that
\begin{equation}\label{idsurfen}
\lim_{\e\to0} J_{\e}(\nu)=\lim_{\e\to 0} I_\e(\nu)\,.
\end{equation}
For $0<\e<1$ we set $\tilde \e=\e/(1-\e)$, and we consider $u_{\tilde
\e}\in\mathcal{B}_{\tilde \e}(a,b,\nu)$ satisfying
$$\int_{Q_\nu}f^{\infty}\left(\frac{x}{\tilde\e},\nabla u_{\tilde\e}\right)dx\leq I_{\tilde \e}(\nu)+ \e\,, $$
where $u_{\tilde \e}(x)=\g_{\tilde \e}(x_\nu/\tilde \e)$ if
$x\in\partial Q_\nu$, for some $\g_{\tilde \e}\in\mathcal{G}(a,b)$.
We define for every $x\in Q_\nu$,
$$v_\e(x):=\begin{cases}
\ds u_{\tilde \e}\left(\frac{x}{1-\e}\right) & \text{if $x\in Q_{\nu,\e}\,$,}\\[10pt]
\ds \g_{\tilde\e}\bigg(\frac{x_\nu}{1-2\|x'\|_{\nu,\infty}}\bigg)
&\text{otherwise}\,.
\end{cases}$$
One may check that $v_\e\in\mathcal{A}_1(a,b,\nu)$, and hence
$$J_\e(\nu)\leq \int_{Q_\nu}f^\infty\left(\frac{x}{\e},\nabla v_\e\right)dx=\int_{Q_{\nu,\e}}f^\infty\left(\frac{x}{\e},\nabla v_\e\right)dx
+\int_{Q_\nu\setminus Q_{\nu,\e}}f^\infty\left(\frac{x}{\e},\nabla
v_\e\right)dx=I_1+I_2\,.$$
We now estimate these two integrals.
First, we have
\begin{equation}\label{2016}I_1=(1-\e)^{N-1}\int_{Q_\nu}f^\infty\left(\frac{y}{\tilde\e},\nabla
u_{\tilde\e}\right)dy \leq (1-\e)^{N-1}\big(I_{\tilde\e}(\nu)+\e\big)\,.\end{equation}
In view of the growth condition (\ref{finfty1gc}),
\begin{align*}
I_2&\leq \beta\int_{Q_\nu\setminus Q_{\nu,\e}}\bigg|\dot \g_{\tilde
\e}\bigg(\frac{x_\nu}{1-2\|x'\|_{\nu,\infty}}\bigg)\bigg|
\bigg(\frac{1}{1-2\|x'\|_{\nu,\infty}}+\frac{|x_\nu||\nabla(\|x'\|_{\nu,\infty})|}{(1-2\|x'\|_{\nu,\infty})^2}\bigg)\,dx\\
&\leq 2\beta\int_{(Q_\nu\setminus Q_{\nu,\e})\cap\{|x_\nu|\leq
(1-2\|x'\|_{\nu,\infty})/2\}}\bigg|\dot \g_{\tilde
\e}\bigg(\frac{x_\nu}{1-2\|x'\|_{\nu,\infty}}\bigg)\bigg|
\bigg(\frac{1}{1-2\|x'\|_{\nu,\infty}}\bigg)\,dx\,,
\end{align*}
where we have used the facts that $\dot \g_{\tilde
\e}(x_\nu/(1-2\|x'\|_{\nu,\infty}))=0$ in the set $\{|x_\nu|>
(1-2\|x'\|_\infty)/2\}$ and
$\|\nabla(\|x'\|_{\nu,\infty})\|_{L^\infty(Q_\nu;\Rb^N)}\leq 1$.
Setting $Q'_\nu=Q_\nu\cap\{x_\nu=0\}$ and
$Q'_{\nu,\e}=Q_{\nu,\e}\cap\{x_\nu=0\}$, we infer from Fubini's
theorem that
\begin{multline}\label{2017}
I_2\leq 2\beta\int_{Q'_\nu\setminus
Q'_{\nu,\e}}\bigg(\int_{-(1-2\|x'\|_{\nu,\infty})/2}^{(1-2\|x'\|_{\nu,\infty})/2}
\bigg|\dot \g_{\tilde
\e}\bigg(\frac{t}{1-2\|x'\|_{\nu,\infty}}\bigg)\bigg|
\bigg(\frac{1}{1-2\|x'\|_{\nu,\infty}}\bigg)\, dt\bigg)\, d\mathcal{H}^{N-1}(x')\leq\\
\leq 2\beta\, \mathcal{H}^{N-1}(Q'_\nu\setminus Q'_{\nu,\e})\,
\mathbf d_\M(a,b)\leq 2\beta \mathbf d_\M(a,b)
\big(1-(1-\e)^{N-1}\big)\,.
\end{multline}
In view of the estimates (\ref{2016}) and (\ref{2017}) obtained for
$I_1$ and $I_2$, we derive that
\begin{equation}\label{2019}
\limsup_{\e\to0}J_\e(\nu)\leq \lim_{\e\to0}
I_\e(\nu)\,.\end{equation}

Conversely, given $0<\e<1$, we consider $\tilde
u_\e\in\mathcal{A}_1(a,b,\nu)$ such that
$$\int_{Q_\nu}f^\infty\left(\frac{x}{\e},\nabla\tilde u_\e\right)dx\leq J_\e(\nu)+\e\,,$$
and $\g\in\mathcal{G}(a,b)$ fixed. We define for $x\in Q_\nu$,
$$w_\e(x):=\begin{cases}
\ds \tilde u_\e\left(\frac{x}{1-\e}\right) &\text{if $x\in Q_{\nu,\e}$,}\\[10pt]
\ds
\g\bigg(\frac{x_\nu}{(1-\e)(2\|x'\|_{\nu,\infty}-1+\e)}\bigg)&\text{otherwise.}
\end{cases}$$
We can check that $w_\e\in\mathcal{B}_{(1-\e)\e}(a,b,\nu)$, and arguing as previously we infer that
\begin{eqnarray*}I_{(1-\e)\e}(\nu)
& \leq & \int_{Q_{\nu,\e}}\!f^\infty\left(\frac{x}{(1-\e)\e}\,,\nabla
w_\e\right)\, dx+ \int_{Q_\nu\setminus
Q_{\nu,\e}}f^\infty\left(\frac{x}{(1-\e)\e}\,,\nabla
w_\e\right)dx\\
&\leq & (1-\e)^{N-1}\big(J_\e(\nu)+\e\big) +2\beta \mathbf d_\M(a,b)\big(1-(1-\e)^{N-1}\big)
\,.
\end{eqnarray*}
Consequently, $\ds\lim_{\e\to0}I_\e(\nu)\leq \liminf_{\e\to0}J_\e(\nu)$, which, together with
(\ref{2019}), completes the proof of Proposition
\ref{limitsurfenerg}.
\prbox
\vskip5pt

We now state the
following properties of the surface energy density.

\begin{proposition}\label{contsurfenerg}
The function $\vartheta_\hom$ is continuous on $\M \times \M \times
\SN$ and there exist constants $C_1>0$  and $C_2>0$ such that
\begin{equation}\label{propsurf}
|\vartheta_{\rm hom}(a_1,b_1,\nu_1) - \vartheta_{\rm hom}(a_2,b_2,\nu_1)| \leq C_1
(|a_1-a_2|+ |b_1-b_2|)\,,
\end{equation}
and
\begin{equation}\label{propsurf2}
\vartheta_{\rm hom}(a_1,b_1,\nu_1)\leq C_2|a_1-b_1|
\end{equation}
for every $a_1,b_1,a_2,b_2\in \M$ and $\nu_1 \in \SN$.
\end{proposition}

\begin{proof}
We use the notation of the previous proof. By Proposition
\ref{limitsurfenerg} together with steps 3 and 4 of the proof of
Proposition \ref{limsurf2}, we get that $\vartheta_\hom(a,b,\cdot)$
is continuous on $\SN$ uniformly with respect to $a$ and $b$. Hence
it is enough to show that (\ref{propsurf}) holds to get the
continuity of $\vartheta_\hom$.
\vskip5pt

{\bf Step 1.} We start with the proof of \eqref{propsurf}. Fix
$\nu_1 \in \SN$ and let $\nu=(\nu_1,\nu_2,\ldots,\nu_N)$ be any
orthonormal basis of $\Rb^N$.
For every $\e>0$, let $\tilde
\e:= \e/(1-\e)$ and consider $\g_{\tilde \e} \in \mathcal
G(a_1,b_1)$ and $u_{\tilde \e} \in \mathcal B_{\tilde
\e}(a_1,b_1,\nu)$ such that $u_{\tilde \e}(x)=\g_{\tilde \e}(x_\nu/\tilde \e)$ for
$x\in\partial Q_\nu$ and
$$\int_{Q_\nu} f^\infty\left(\frac{x}{\tilde \e},\nabla u_{\tilde
\e}\right) dx \leq I_{\tilde \e}(a_1,b_1,\nu) + \e\,.$$
We shall now
carefully modify $u_{\tilde \e}$ in order to
get another function $v_\e \in \mathcal A_1(a_2,b_2,\nu)$. We will
proceed as in the proofs of Propositions \ref{limitsurfenerg} and
\ref{limsurf2}. Let $\g_a \in \mathcal G(a_2,a_1)$ and $\g_b \in
\mathcal G(b_2,b_1)$, and define
$$v_\e(x):=\left\{
\begin{array}{lll}
\ds u_{\tilde\e}\left(\frac{x}{1-\e}\right) & \text{ if } & x \in
Q_{\nu,\e}\,,\\[0.3cm]
\ds \g_{\tilde\e}\left(\frac{x_\nu}{1-2\|x'\|_{\nu,\infty}} \right)
& \text{ if } & \ds x \in A_1\,,\\[0.3cm]
\ds \g_a\left(\frac{2\|x\|_{\nu,\infty}-1}{\e}+\frac{1}{2} \right) &
\text{ if } & \ds  x \in A_2:=(Q_\nu\setminus
Q_{\nu,\e})\cap\{x_\nu\geq \e/2\}\,,\\[0.3cm]
\ds \g_b\left(\frac{2\|x\|_{\nu,\infty}-1}{\e}+\frac{1}{2} \right) &
\text{ if } & \ds  x \in A_3:=(Q_\nu\setminus Q_{\nu,\e})\cap\{x_\nu\leq -\e/2\}\,,\\[0.3cm]
\ds \g_a\left(\frac{2\|x'\|_{\nu,\infty}-1}{2x_\nu}+\frac{1}{2}
\right) & \text{ if } & \ds x \in A_4:=\left\{ 0 < x_\nu \leq
\frac{\e}{2}\,,\;
\frac{1}{2}-x_\nu \leq \|x'\|_{\nu,\infty} < \frac{1}{2}\right\}\,,\\[0.3cm]
\ds \g_b\left(\frac{1-2\|x'\|_{\nu,\infty}}{2x_\nu}+\frac{1}{2}
\right) & \text{ if } & \ds x \in A_5:=\left\{ -\frac{\e}{2} < x_\nu
\leq 0\,,\; \frac{1}{2}+x_\nu \leq \|x'\|_{\nu,\infty} <
\frac{1}{2}\right\}\,,
\end{array}\right.$$
with
$$ A_1:=\left\{ \frac{1-\e}{2} \leq
\|x'\|_{\nu,\infty} < \frac{1}{2}
\text{ and } |x_\nu|\leq -\|x'\|_{\nu,\infty}+\frac{1}{2}\right\}\,.$$
One may check that the function $v_\e$ has been constructed in such
a way that $v_\e \in \A_1(a_2,b_2,\nu)$, and thus
\begin{equation}\label{Je}J_\e(a_2,b_2,\nu) \leq
\int_{Q_\nu}f^\infty\left(\frac{x}{\e},\nabla v_\e\right) dx\,.
\end{equation}
Arguing exactly as in the proof of Proposition \ref{limsurf2}, one
can show that
\begin{equation}\label{cunu}\int_{Q_{\nu,\e}}
f^\infty\left(\frac{x}{\e},\nabla v_\e\right) dx \leq I_{\tilde
\e}(a_1,b_1,\nu) + \e\,,\end{equation}
and
\begin{eqnarray}\label{A1}\int_{A_1}
f^\infty\left(\frac{x}{\e},\nabla v_\e\right) dx \leq C\mathbf
d_\M(a_1,b_1)(1-(1-\e)^{N-1}))\,.
\end{eqnarray}
Now we only estimate the integrals over $A_2$ and $A_4$, the ones
over $A_3$ and $A_5$ being very similar. Define the Lipschitz
function $F_\e:\Rb^{N} \to \Rb$ by
$$F_\e(x):=\frac{2\|x\|_{\nu,\infty}-1}{\e}+\frac{1}{2}\,.$$
Using the growth condition (\ref{finfty1gc}) together with Fubini's
theorem, and the fact that $A_2 \subset
F_\e^{-1}\big([-1/2,1/2)\big)$, we derive
\begin{multline*}
\int_{A_2} f^\infty\left(\frac{x}{\e},\nabla v_\e\right) dx  \leq
\b \int_{A_2} |\dot\g_a (F_\e(x))|\, |\nabla F_\e(x)|\, dx
 \leq\\
 \leq  \b \int_{F_\e^{-1}([-1/2,1/2))} |\dot\g_a
(F_\e(x))|\, |\nabla F_\e(x)|\, dx
 \leq  \b \int_{-1/2}^{1/2}|\,\dot \g_a(t)| \,
\HH^{N-1}(F_\e^{-1}\{t\})\, dt
\,,
\end{multline*}
where we used the Coarea
formula in the last inequality.
We observe that for every $t\in(-1/2,1/2)$, $F^{-1}_\e\{t\}=\partial Q_{\nu,\frac{\e(1-2t)}{2}}$ so that
$\HH^{N-1}(F_\e^{-1}\{t\})\leq \HH^{N-1}(\partial Q)$. Therefore
\begin{eqnarray}\label{A2}\int_{A_2}
f^\infty\left(\frac{x}{\e},\nabla v_\e\right) dx \leq \b
\HH^{N-1}(\partial Q)\mathbf d_\M(a_1,a_2)\,.\end{eqnarray} Define
now $G : \Rb^N\setminus\{x_\nu=0\} \to \Rb$ by
$$G(x):=\frac{2\|x'\|_{\nu,\infty}-1}{2x_\nu}+\frac{1}{2}\,.$$
The growth condition (\ref{finfty1gc}) and Fubini's theorem yield
\begin{multline*}
\int_{A_4} f^\infty\left(\frac{x}{\e},\nabla v_\e\right) dx \,\leq \\
\leq \b
\int_{0}^{\e/2} \left(\int_{G(\cdot,x_\nu)^{-1}([-1/2,1/2))}
|\dot\g_a (G(x',x_\nu))|\, |\nabla G(x',x_\nu)|\, d\HH^{N-1}(x')
\right)dx_\nu\,.
\end{multline*}
As $|\nabla_{x'} G(x)| = 1/x_\nu$ and
$|\nabla_{x_\nu} G(x)| \leq 1/x_\nu$ for a.e. $x \in A_4$, it
follows that $|\nabla G(x)| \leq 2 |\nabla_{x'} G(x)|$ for a.e.  $x
\in A_4$. Hence
\begin{multline*}
\int_{A_4} f^\infty\left(\frac{x}{\e},\nabla v_\e\right) dx \,\leq \\
\leq 2\b
\int_{0}^{\e/2} \left(\int_{G(\cdot,x_\nu)^{-1}([-1/2,1/2))}
|\dot\g_a (G(x',x_\nu))|\, |\nabla_{x'} G(x',x_\nu)|\,
d\HH^{N-1}(x') \right)dx_\nu\,.
\end{multline*}
For every $x_\nu \in (0,\e/2)$ the
function $G(\cdot,x_\nu):\Rb^{N-1} \to \Rb$ is Lipschitz, and thus the
Coarea formula implies
\begin{eqnarray}\label{A4}
\int_{A_4} f^\infty\left(\frac{x}{\e},\nabla v_\e\right) dx & \leq &
2\b \int_0^{\e/2} \left( \int_{-1/2}^{1/2}|\, \dot\g_a(t)|\,
\HH^{N-2}(\{x': G(x',x_\nu)=t\})\, dt \right)dx_\nu\nonumber\\
&\leq & C\e\, \mathbf d_\M(a_1,a_2)\,,
\end{eqnarray}
where we used as previously  the estimate  $\HH^{N-2}(\{x': G(x',x_\nu)=t\})\leq \HH^{N-2}\big(\partial (\frac{-1}{2},\frac{1}{2})^{N-1}\big)$.
Gathering (\ref{Je}) to (\ref{A4}) and considering the analogous estimates
for the integrals over $A_3$ and $A_5$ (with $b_1$ and $b_2$ instead
of $a_1$ and $a_2$), we infer that
$$J_\e(a_2,b_2,\nu)
\leq \int_{Q_\nu}f^\infty\left(\frac{x}{\e},\nabla v_\e\right) dx
\leq I_{\tilde\e}(a_1,b_1,\nu) + C\big(\e + \mathbf d_\M(a_1,a_2) +
\mathbf d_\M(b_1,b_2)\big)\,.$$ Taking the limit as $\e \to 0$, we
get in light of Propositions \ref{limitsurfenerg} and \ref{limsurf2}
that
$$\vartheta_\hom(a_2,b_2,\nu) \leq \vartheta_\hom(a_1,b_1,\nu) + C \big( \mathbf
d_\M(b_1,b_2) + \mathbf d_\M(a_1,a_2) \big)\,.$$ Since the geodesic
distance on $\M$ is equivalent to the Euclidian distance, we
conclude, possibly exchanging the roles of $(a_1,b_1)$ and
$(a_2,b_2)$, that (\ref{propsurf}) holds. \vskip5pt

{\bf Step 2.} We now prove \eqref{propsurf2}. Given an arbitrary orthonormal basis $\nu=(\nu_1,\ldots,\nu_N)$ of $\Rb^N$,  let $\g \in
\mathcal G(a_1,b_1)$ and define $u_\e(x):=\g(x_\nu/\e)$. Obviously
$u_\e \in \mathcal B_\e(a_1,b_1,\nu)$. Using (\ref{idsurfen})
together with the growth condition (\ref{finfty1gc}) satisfied by
$f^\infty$, we derive that
$$\vartheta_\hom(a_1,b_1,\nu_1) \leq \liminf_{\e \to
0}\int_{Q_\nu}f^\infty\left(\frac{x}{\e},\nabla u_\e\right)dx \leq
\liminf_{\e \to 0} \frac{\b}{\e}\int_{Q_\nu}
\left|\dot\g\left(\frac{x\cdot\nu_1}{\e}\right)\right|\,
dx=\b\mathbf d_{\M}(a_1,b_1)\,.$$ Then (\ref{propsurf2}) follows
from the equivalence between  $\mathbf d_{\M}$ and the Euclidian
distance.
\end{proof}

\section{Localization and integral repersentation on partitions}\label{BV}

In this section we first show that the $\Gamma$-limit defines a measure. Then we prove an abstract representation on
partitions in sets of finite perimeter.  This two facts will allow us to obtain the upper bound on the $\Gamma$-limit in the next section.

\subsection{Localization}

We consider an arbitrary given sequence
$\{\e_n\} \searrow 0^+$ and we localize the functionals
$\{\F_{\e_n}\}_{n\in\Nb}$ on the family $\A(\O)$, {\it i.e.}, for
every $u \in L^1(\O;\Rb^d)$ and every $A \in \A(\O)$, we set
$$\F_{\e_n}(u,A):= \begin{cases}
\ds \int_A
f\left(\frac{x}{\e_n},\nabla u\right) dx & \text{if }u \in
W^{1,1}(A;\M)\,,\\[8pt]
+\infty & \text{otherwise}\,.
\end{cases}$$
Next we define for $u\in L^1(\O;\Rb^d)$ and  $A \in \A(\O)$,
\begin{equation*}
\F(u,A):= \inf_{\{u_n\}} \bigg\{ \liminf_{n \to +\infty}\, \F_{\e_n}
(u_n,A)\, :\, u_n \to u \text{ in }L^1(A;\Rb^d) \bigg\}\,.
\end{equation*}
Note that $\F(u,\cdot)$ is an increasing set function for every
$u\in L^1(\O;\Rb^d)$ and that $\F(\cdot,A)$ is  lower semicontinuous
with respect to the strong  $L^1(A;\Rb^d)$-convergence for every
$A\in \A(\O)$.

Since $L^1(A;\Rb^d)$ is separable, \cite[Theorem~8.5]{DM} and a
diagonalization argument bring the existence of a subsequence (still
denoted $\{\e_n\}$) such that $\F(\cdot,A)$ is the $\G$-limit of
$\F_{\e_n}(\cdot,A)$ for the strong $L^1(A;\Rb^d)$-topology for
every $A \in \R(\O)$ (or $A=\O$). \vskip5pt

We have the following locality property of the
$\G$-limit which, in the $BV$ setting, parallels \cite[Lemma 3.1]{BM}.

\begin{lemma}\label{measbis}
For every $u \in BV(\O;\M)$, the set function $\F(u,\cdot)$ is the
restriction to $\A(\O)$ of a Radon measure absolutely continuous
with respect to $\LL^N+|Du|$.
\end{lemma}

\begin{proof}
Let $u \in BV(\O;\M)$ and $A \in \A(\O)$. By Theorem 3.9 in
\cite{AFP}, there exists a sequence $\{u_n\} \subset
W^{1,1}(A;\Rb^d) \cap \C^\infty(A;\Rb^d)$ such that $u_n \to u$ in
$L^1(A;\Rb^d)$ and $\int_A |\nabla u_n|\, dx \to |Du|(A)$. Moreover,
$u_n(x)\in {\rm co}(\M)$ for a.e. $x \in A$ and every $n \in \Nb$.
Applying Proposition \ref{proj} to $u_n$, we obtain a new sequence
$\{w_n\} \subset W^{1,1}(A;\M)$ satisfying
$$\int_A |\nabla w_n|\, dx \leq C_\star \int_A |\nabla u_n|\, dx,$$ for
some constant $C_\star>0$ depending only on $\M$ and $d$. From
construction of $w_n$, we have that $w_n \to u$ in $L^1(A;\Rb^d)$.
Taking $\{w_n\}$ as admissible sequence, we deduce in light of the
growth condition $(H_2)$ that
\begin{equation*}
\F(u,A) \leq \beta\big(\LL^N(A)+C_\star|Du|(A)\big)\,.
\end{equation*}

We now prove  that
\begin{equation*}
\F(u,A) \leq \F(u,B) + \F(u,A \setminus \overline C)
\end{equation*}
for every $A$, $B$ and $C \in \A(\O)$ satisfying $\overline C
\subset B \subset A$. Then the measure property of
$\F(u,\cdot)$ can be obtained as in the proof of \cite[Lemma 3.1]{BM} with
minor modifications.  For this reason, we  shall omit it.

Let $R \in \R(\O)$ such that $C \subset\subset R \subset\subset B$
and consider $\{u_n\} \subset W^{1,1}(R;\M)$ satisfying $u_n \to u$
in $L^1(R;\Rb^d)$ and
\begin{equation}\label{u_nbis}
\lim_{n \to +\infty} \F_{\e_n}(u_n,R) = \F(u,R)\,.
\end{equation}
Given $\eta>0$ arbitrary, there exists a sequence $\{v_n\} \subset
W^{1,1}(A \setminus \overline C;\M)$ such that $v_n \to u$ in
$L^1(A\setminus \overline C;\Rb^d)$ and
\begin{equation}\label{v_nbis}
\liminf_{n \to +\infty}\, \F_{\e_n}(v_n,A \setminus \overline C) \leq
\F(u,A \setminus \overline C) + \eta\,.
\end{equation}
By Theorem \ref{density}, we can assume without loss of generality
that $u_n \in \D(R;\M)$ and $v_n \in \D(A \setminus \overline
C;\M)$. Let $L:=\dist(C,\partial R)$ and define for every $i \in
\{0,\ldots,n\}$,
$$R_i:=\bigg\{x \in R:\, \dist(x,\partial R)
>\frac{iL}{n}\bigg\}\,.$$
Given $i \in \{0,\ldots,n-1\}$, let $S_i:=R_i \setminus
\overline{R_{i+1}}$ and consider a cut-off function $\zeta_i \in
\C^\infty_c(\O;[0,1])$ satisfying $\zeta_i(x) = 1$ for $x\in
R_{i+1}$, $\zeta_i(x) = 0$ for $x\in \O\setminus R_{i}$ and $|\nabla
\zeta_i|\leq 2n/L$. Define
$$z_{n,i}:=\zeta_i u_n + (1-\zeta_i)v_n \in W^{1,1}(A;\Rb^d)\,.$$
If $\pi_1(\M) \neq 0$, $z_{n,i}$ is smooth in $A\setminus
\Sigma_{n,i}$ with $\Sigma_{n,i} \in \mathcal S$, while $z_{n,i}$ is
smooth in $A$ if $\pi_1(\M) = 0$. Observe that $z_{n,i}(x) \in {\rm
co}(\M)$ for a.e. $x \in A$ and actually, $z_{n,i}$ fails to be
$\M$-valued exactly in the set $S_i$. To get an admissible sequence,
we project $z_{n,i}$ on $\M$ using Proposition \ref{proj}. It yields
a sequence $\{w_{n,i}\} \subset W^{1,1}(A;\M)$ satisfying $w_{n,i}
=z_{n,i}$ a.e. in $A \setminus S_i$,
\begin{equation}\label{L1}
\int_A|w_{n,i} -u |\, dx \leq \int_A |z_{n,i} - u|\, dx+C
\LL^N(S_i)\,,
\end{equation}
for some constant $C>0$ depending only on the diameter of ${\rm
co}(\M)$, and
$$\int_{S_i} |\nabla w_{n,i}|\, dx  \leq C_\star \int_{S_i} |\nabla z_{n,i}|\, dx
 \leq C_\star \int_{S_i}\left(|\nabla u_n| + |\nabla v_n| +
\frac{n}{2L} |u_n - v_n|\right)\, dx\,.$$ Arguing exactly as in the
proof of \cite[Lemma 3.1]{BM}, we now find an index $i_n \in
\{0,\ldots,n-1\}$ such that
\begin{multline}\label{nin}
\F_{\e_n}(w_{n,i_n},A)  \leq
\F_{\e_n}(u_n,R)+\F_{\e_n}(v_n,A \setminus \overline C)\,+\\
+ C_0 \int_{R \setminus \overline C}|u_n - v_n|\, dx +
\frac{C_0}{n}\sup_{k \in \Nb}\int_{R \setminus \overline C}(1+|\nabla
u_k|+ |\nabla v_k|)\, dx\,,
\end{multline}
for some constant $C_0$ independent of $n$.

A well
known consequence of the Coarea formula yields (see, {\it e.g.},
\cite[Lemma 3.2.34]{Federer}),
\begin{equation}\label{vanslice}
\LL^N(S_{i_n})  =  \int_{i_n
L/n}^{(i_n+1)L/n} \HH^{N-1}(\{x \in R: \dist(x,\partial R)=t\})\, dt \to 0 \quad\text{as $n \to
+\infty$\,.}
\end{equation}
As a consequence of  (\ref{L1}) and \eqref{vanslice},  $w_{n,i_n}
\to u$ in $L^1(A;\Rb^d)$. Taking the $\liminf$ in
(\ref{nin}) and using (\ref{u_nbis}) together with (\ref{v_nbis}), we derive
$$\F(u,A) \leq \F(u,R) + \F(u,A \setminus \overline C)+\eta \leq \F(u,B) + \F(u,A \setminus \overline C)+\eta\,.$$
The conclusion follows  from the arbitrariness of $\eta$.
\end{proof}

\begin{remark}\label{measconstr}
In view of Lemma \ref{measbis}, for every $u\in BV(\O;\M)$, the set
function  $\F(u,\cdot)$ can be uniquely extended to a Radon measure
on $\O$. Such a measure is given by
$$\F(u,B):= \inf \big\{\F(u,A)\,:\,A\in\A(\O),\,B\subset A\big\}\,, $$
for every $B\in\mathcal{B}(\O)$ (see, \emph{e.g.}, \cite[Theorem
1.53]{AFP}).
\end{remark}

\subsection{Integral representation on partitions}

Besides the locality of $\F(u,\cdot)$, another key point of the
analysis is to prove an abstract integral representation on
partitions. Similarly to {\it e.g.} \cite[Lemma 3.7]{BDV}, using
$(H_1)$ we easily obtain the translation invariance property of the
$\G$-limit, the proof of which is omitted.

\begin{lemma}\label{invtrans}
For every $u \in BV(\O;\M)$, every $A \in \A(\O)$ and every
$y\in\Rb^N$ such that $y+A \subset \O$, we have
$$\F(\tau_yu,y+A)=\F(u,A) \,,$$
where $(\tau_yu)(x):=u(x-y)$.
\end{lemma}

We are now in position to prove the integral representation of
the $\G$-limit on partitions. 

\begin{proposition}\label{reppart}
There exists a unique function $K:\M\times\M\times\SN\to[0,+\infty)$
continuous in the last variable and such that
\newline
\emph{(i)} $K(a,b,\nu)=K(b,a,-\nu)$ for every
$(a,b,\nu)\in\M\times\M\times\SN$,
\newline
\emph{(ii)} for every finite subset $T$ of $\M$,
\begin{equation}\label{intsurfbor}
\F(u,S)=\int_{S} K(u^+,u^-,\nu_u)\, d\HH^{N-1}\,,
\end{equation}
for every $u\in BV(\O;T)$ and every Borel subset $S$ of $\O\cap S_u\,$.
\end{proposition}

\begin{proof}
It follows the argument of \cite[Proposition 4.2]{BDV} that is based
on the general result \cite[Theorem 3.1]{AB}, on account to Lemmas
\ref{measbis}, \ref{invtrans} and Remark \ref{measconstr}. We omit
any further details.
\end{proof}

\section{The upper bound}

\noindent We now adress the $\G$-$\limsup$ inequality. The upper
bound on the diffuse part will be obtained using an extension of the
relaxation result of \cite{AEL} (see Theorem \ref{relax} in the
Appendix) together with the partial representation of the $\G$-limit
already established in $W^{1,1}$ (see Theorem~\ref{babmilp=1}). The
estimate of the jump part relies on the integral representation on
partitions in sets of finite perimeter stated in Proposition
\ref{reppart}. \vskip5pt

In view of the measure property of the $\G$-limit, we may write for
every $u\in BV(\O;\M)$,
\begin{equation}\label{decompupbd}
\F(u,\O)=\F(u,\O\setminus S_u)+\F(u,\O\cap S_u)\,.
\end{equation}
Hence the desired upper bound $\F(u,\O)\leq \F_{\rm hom}(u)$ will follow estimating separately the two terms in the right handside of \eqref{decompupbd}.

\begin{lemma}\label{upperboundBV}
For every $u \in BV(\O;\M)$, we have
$$\F(u,\O\setminus S_u) \leq \int_\O Tf_{\rm hom}(u,\nabla u)\,
dx+ \int_\O Tf^\infty_{\rm hom}\left(\tilde
u,\frac{dD^cu}{d|D^cu|}\right)\, d|D^cu|\,.$$
\end{lemma}

\begin{proof}
Let $A \in \A(\O)$ and $\{u_n\} \subset W^{1,1}(A;\M)$ be such that
$u_n \to u$ in $L^1(A;\Rb^d)$. Since $\F(\cdot,A)$ is sequentially
lower semicontinuous for the strong $L^1(A;\Rb^d)$ convergence, it
follows from Theorem~\ref{babmilp=1} that
$$\F(u,A) \leq \liminf_{n \to +\infty}\, \F(u_n,A)=\liminf_{n \to +\infty}\, \int_A Tf_\hom(u_n,\nabla u_n)\, dx\,.$$
Since the  sequence $\{u_n\}$ is arbitrary, we deduce
$$\F(u,A) \leq \inf\left\{\liminf_{n \to +\infty} \int_A Tf_\hom(u_n,\nabla u_n)\, dx : \{u_n\} \subset W^{1,1}(A;\M),
\, u_n \to u\text{ in }L^1(A;\Rb^d)\right\}.$$ According to
Proposition \ref{properties1}, the energy density $Tf_\hom$ is a
continuous and tangentially quasiconvex function which fulfills the
assumptions of Theorem \ref{relax}. Hence
\begin{equation}\label{FUA}
\F(u,A) \leq \int_A Tf_\hom(u,\nabla u)\,
dx+ \int_A Tf^\infty_\hom\left(\tilde u,\frac{dD^cu}{d|D^cu|}\right)
d|D^cu| + \int_{S_u \cap A} H(u^+,u^-,\nu_u)\, d\HH^{N-1}
\end{equation}
for some function $H:\M \times \M \times \SN \to [0,+\infty)$. By
outer regularity, (\ref{FUA}) holds for every $A \in \mathcal
B(\O)$. Taking $A=\O \setminus S_u$, we obtain
$$\F(u,\O \setminus S_u) \leq \int_\O Tf_\hom(u,\nabla u)\,
dx+ \int_\O Tf^\infty_\hom\left(\tilde
u,\frac{dD^cu}{d|D^cu|}\right)\, d|D^cu|\,,
$$ and the proof is complete.
\end{proof}

To prove the upper bound of the jump part, we first need to compare the energy density $K$ obtained in Proposition \ref{reppart}
with the expected density $\vartheta_\hom$.

\begin{lemma}\label{upbdsurf}
We have $K(a,b,\nu_1)\leq \vartheta_{\rm hom}(a,b,\nu_1)$ for every
$(a,b,\nu_1)\in \M\times\M\times\mathbb{S}^{N-1}$.
\end{lemma}

\begin{proof}
We will partially proceed  as in the proof of Proposition
\ref{limsurf2} and we refer to it for the notation. Consider
$\nu=(\nu_1,\ldots,\nu_N)$ an orthonormal basis of $\Rb^N$. We shall
prove that $K(a,b,\nu_1)\leq \vartheta_{\rm hom}(a,b,\nu_1)$. Since
$K$ and $\vartheta_{\rm hom}$ are continuous in the last variable,
we may assume that $\nu$ is a rational basis, {\it i.e.}, for all $i
\in \{1,\ldots,N\}$, there exists $\g_i \in \Rb\setminus \{0\}$ such
that $v_i:= \g_i \nu_i \in \Zb^N$, and the general case follows by
density. \vskip5pt

Given $0<\eta<1$ arbitrary, by Proposition \ref{limitsurfenerg} and
\eqref{idsurfen} we can find $\e_0>0$,
$u_0\in\mathcal{B}_{\e_0}(a,b,\nu)$ and
$\gamma_{\e_0}\in\mathcal{G}(a,b)$ such that
$u_0(x)=\gamma_{\e_0}(x\cdot\nu_1/\e_0)$ and
$$\int_{Q_{\nu}} f^{\infty}\bigg(\frac{x}{\e_0},\nabla u_0\bigg)\, dx\leq \vartheta_{\rm hom}(a,b,\nu_1)+\eta\,.$$
For every $\lambda=(\lambda_2,\ldots,\lambda_N)\in\Zb^{N-1}$, we set
$x_n^{(\lambda)}:=\e_n\sum_{i=2}^N\lambda_iv_i$ and
$Q_{\nu,n}^{(\lambda)}:=x^{(\lambda)}_n+(\e_n/\e_0)Q_\nu$. We define
the set $\Lambda_n$ by
\begin{multline*}
\Lambda_n
:=\Bigg\{\lambda\in\Zb^{N-1}\;:\;Q_{\nu,n}^{(\lambda)}\subset
Q_{\nu} \text{ and } x_n^{(\lambda)}\in \sum_{i=2}^N
l_i\left(\frac{\e_n}{\e_0}+\e_n \gamma_i\right)\nu_i+\e_n  P\\
\text{ for some }(l_2,\ldots,l_N)\in\Zb^{N-1}\Bigg\}\,,
\end{multline*}
where $$P:=\Big\{\a_2v_2 + \ldots + \a_N v_N : \, \a_2,\ldots,\a_N
\in [-1/2,1/2)\Big\}.$$ Next consider
$$u_n(x)=\begin{cases}
\ds u_0\bigg(\frac{\e_0(x-x_n^{(\lambda)})\cdot\nu_1}{\e_n}\bigg) & \text{if $x\in Q_{\nu,n}^{(\lambda)}$ for some $\lambda\in\Lambda_n$}\,,\\[10pt]
\ds \gamma_{\e_0}\bigg(\frac{x\cdot\nu_1}{\e_n}\bigg) & \text{otherwise}\,.
\end{cases}$$
Note that $u_n\in W^{1,1}(Q_\nu;\M)$, $\{\nabla u_n\}$ is bounded in
$L^1(Q_\nu;\Rb^{d \times N})$, and $u_n\to u^{a,b}_{\nu_1}$ in
$L^1(Q_\nu;\Rb^d)$ as $n\to+\infty$ with $u^{a,b}_{\nu_1}$ given
by \begin{equation*}
u_{\nu_1}^{a,b}(x):=\begin{cases}
a & \text{if } x\cdot\nu_1 \geq 0\,, \\
b & \text{if } x\cdot\nu_1 <0\,,
\end{cases}
\quad \Pi_{\nu_1}:=\big\{x\in\Rb^N : x\cdot\nu_1=0\big\}\,.
\end{equation*}
Arguing as in Step 1 of the proof of \cite[Proposition 2.2]{BDV}, we
obtain that 
\begin{equation}\label{pasidee1}
\limsup_{n\to+\infty}\,
\int_{Q_\nu}f^\infty\bigg(\frac{x}{\e_n},\nabla u_n\bigg)\, dx\leq
\int_{Q_\nu} f^\infty\bigg(\frac{x}{\e_0},\nabla u_0\bigg)\, dx \leq
\vartheta_{\rm hom}(a,b,\nu_1)+\eta\,.
\end{equation}
For $\rho>0$  define $A_\rho:=Q_\nu\cap\{|x\cdot\nu_1|<\rho\}$. By
construction the sequence $\{u_n\}$ is admissible for
$\F(u^{a,b}_{\nu_1},A_\eta)$ so that
\begin{multline}\label{pasidee1b}
\F\big(u^{a,b}_{\nu_1},A_\eta\cap\Pi_{\nu_1}\big)\leq\F(u^{a,b}_{\nu_1},A_\eta)\leq
\liminf_{n\to+\infty} \,\int_{A_\eta}f\bigg(\frac{x}{\e_n},\nabla u_n\bigg)\, dx\leq \\
\leq\beta
\mathcal{L}^N(A_\eta)+\liminf_{n\to+\infty}\,\int_{A_{\e_n}}f\bigg(\frac{x}{\e_n},\nabla
u_n\bigg)\, dx\leq
\liminf_{n\to+\infty}\,\int_{A_{\e_n}}f\bigg(\frac{x}{\e_n},\nabla
u_n\bigg)\, dx +\beta\eta\,,
\end{multline}
where we have used $(H_2)$ and the fact that $\nabla u_n=0$ outside
$A_{\e_n}$. On the other hand, Proposition~\ref{reppart} yields
\begin{equation}\label{pasidee2}
\F\big(u^{a,b}_{\nu_1},A_\eta\cap\Pi_{\nu_1}\big)=\int_{A_\eta\cap\Pi_{\nu_1}}K(a,b,\nu_1)\,
d\HH^{N-1}= K(a,b,\nu_1)\,.
\end{equation}
Using $(H_4)$, the boundedness of $\{\nabla u_n\}$ in
$L^1(Q_\nu;\Rb^{d \times N})$, the fact that
$f^\infty(\cdot,0)\equiv 0$, and  H\"older's inequality, we derive
\begin{eqnarray}\label{pasidee3}
\bigg|\int_{A_{\e_n}}f\bigg(\frac{x}{\e_n},\nabla u_n\bigg) \,dx-
\int_{Q_\nu}f^\infty\bigg(\frac{x}{\e_n},\nabla u_n\bigg)\,dx\bigg|& \leq & C\int_{A_{\e_n}}(1+|\nabla u_n|^{1-q})\,dx\nonumber\\
& \leq & C\big(\e_n+\e_n^q\|\nabla u_n\|^{1-q}_{L^1(Q_\nu;\Rb^{d
\times N})}\big)\to 0\,
\end{eqnarray}
as $n \to \infty$. Gathering \eqref{pasidee1}, \eqref{pasidee1b}, \eqref{pasidee2} and
\eqref{pasidee3}, we obtain $K(a,b,\nu_1)\leq \vartheta_{\rm
hom}(a,b,\nu_1)+(\beta+1)\eta$ and the conclusion follows from the
arbitrariness of $\eta$.
\end{proof}

We are now in position to prove the upper bound on the jump part of the energy. The argument is based on
Lemma \ref{upbdsurf}  together with an approximation procedure of \cite{AMT}. In view of Lemma \ref{upperboundBV}
and \eqref{decompupbd}, this will complete the proof of the upper bound $\F(u,\O)\leq \F_{\rm hom}(u)$.

\begin{corollary}\label{upbdjp}
For every $u\in BV(\O;\M)$, we have
$$\F(u,\O\cap S_u)\leq \int_{\O\cap S_u} \vartheta_{\rm hom}(u^+,u^-,\nu_u)\, d\HH^{N-1}\,.$$
\end{corollary}

\begin{proof} First assume that $u$ takes a finite number of values, {\it i.e.}, $u\in BV(\O;T)$ for some finite subset $T\subset \M$.
Then the conclusion directly follows from Proposition \ref{reppart} together with Lemma~\ref{upbdsurf}.

Fix an arbitrary function $u\in BV(\O;\M)$ and an open set
$A\in\A(\O)$. For $\delta_0>0$ small enough, let  $\mathcal
U:=\big\{s \in\Rb^d\,:\,\dist(s,\M)<\d_0\big\}$ be the
$\d_0$-neighborhood of $\M$ on which the nearest point projection
$\Pi: \mathcal U \to \M$ is a well defined Lipschitz mapping. We
extend $\vartheta_{\rm hom}$ to a function $\hat \vartheta_{\rm
hom}$ defined in $\Rb^d \times \Rb^d
\times\mathbb{S}^{N-1}$ by setting
$$\hat \vartheta_{\rm hom}(a,b,\nu):=\chi(a)\chi(b)\vartheta_{\rm hom}\bigg(\Pi(a),\Pi(b), \nu\bigg)\,, $$
for a cut-off function $\chi
\in \C^\infty_c(\Rb^d;[0,1])$ satisfying $\chi(t)=1$ if $\dist(s,\M)
\leq \delta_0/2$, and $\chi(s)=0$ if $\dist(s,\M) \geq 3\delta_0/4$.
In view of Proposition \ref{contsurfenerg}, we infer that
$\hat\vartheta_{\rm hom}$ is continuous and satisfies
$$|\hat\vartheta_{\rm hom}(a_1,b_1,\nu)-\hat\vartheta_{\rm hom}(a_2,b_2,\nu)|\leq C\big(|a_1-a_2|+|b_1-b_2|\big)\,,$$
and
$$\hat\vartheta_{\rm hom}(a_1,b_1,\nu)\leq  C|a_1-b_1|\,,$$
for every $a_1$, $b_1$, $a_2$, $b_2\in\Rb^d$, $\nu\in\SN$, and
some constant $C>0$. Therefore we can apply Step~2 in the proof of
\cite[Proposition 4.8]{AMT} to obtain a sequence $\{v_n\}\subset
BV(\O;\Rb^d)$ such that, for every $n\in \Nb$, $v_n\in BV(\O;T_n)$
for some finite set $T_n\subset \Rb^d$, $v_n\to u$ in
$L^\infty(\O;\Rb^d)$ and
\begin{eqnarray*}
\limsup_{n\to+\infty}\,\int_{A\cap S_{v_n}}\hat\vartheta_{\rm
hom}(v_n^+,v_n^-,\nu_{v_n})\, d\HH^{N-1} & \leq & C|Du|(A\setminus
S_u)+
\int_{A\cap S_{u}}\hat\vartheta_{\rm hom}(u^+,u^-,\nu_{u})\, d\HH^{N-1}\\
& = & C|Du|(A\setminus S_u)+ \int_{A\cap S_{u}}\vartheta_{\rm
hom}(u^+,u^-,\nu_{u})\, d\HH^{N-1}\,.
\end{eqnarray*}
Hence we may assume without loss of generality
that for each $n\in \Nb$, $\|v_n - u\|_{L^\infty(\O;\Rb^d)} < \d_0/2$,
and thus $\dist(v_n^\pm(x),\M) \leq |v_n^\pm(x)-u^\pm(x)|< \d_0/2$ for
$\HH^{N-1}$-a.e. $x \in S_{v_n}$. In particular, we can define
$$u_n:=\Pi(v_n)\,, $$
and then $u_n\in BV(\O;\M)$, $u_n\to u$ in $L^1(\O;\Rb^d)$.
Moreover, one may check that for each $n\in\mathbb{N}$,
$S_{u_n}\subset S_{v_n}$ so that
$\HH^{N-1}\big(S_{u_n}\setminus(J_{u_n}\cap J_{v_n})\big)\leq
\HH^{N-1}(S_{u_n}\setminus J_{u_n})+ \HH^{N-1}(S_{v_n}\setminus
J_{v_n})=0$, and
$$u_n^\pm(x)=\Pi(v_n^\pm(x)) \quad \text{ and } \quad \nu_{u_n}(x)=\nu_{v_n}(x)
\quad\text{ for every $x\in J_{u_n}\cap J_{v_n}$}\,.$$
Consequently,
\begin{multline}\label{estilimsurf}
\limsup_{n\to+\infty}\, \int_{A\cap S_{u_n}}\vartheta_{\rm
hom}(u_n^+,u_n^-,\nu_{u_n})\, d\HH^{N-1}  \,\leq \\
\leq \limsup_{n\to+\infty}\int_{A\cap S_{v_n}}
\hat\vartheta_{\rm hom}(v_n^+,v_n^-,\nu_{v_n})\, d\HH^{N-1}\,\leq\\
\leq  C|Du|(A\setminus S_u)+ \int_{A\cap S_{u}}\vartheta_{\rm
hom}(u^+,u^-,\nu_{u})\, d\HH^{N-1}\,.
\end{multline}
Since $u_n$ takes a finite number of values, Proposition
\ref{reppart} and Lemma \ref{upbdsurf} yield 
\begin{equation}\label{estisurfn}
\F(u_n,A\cap S_{u_n})\leq \int_{A\cap S_{u_n}}\vartheta_{\rm
hom}(u_n^+,u_n^-,\nu_{u_n})\, d\HH^{N-1}\,,
\end{equation}
and, in view of Lemma  \ref{measbis},
\begin{equation}\label{estidiffn}
\F(u_n,A\setminus S_{u_n}) \leq C\LL^N(A)\,.
\end{equation}
Combining \eqref{estilimsurf}, \eqref{estisurfn} and \eqref{estidiffn}, we deduce
\begin{eqnarray*}
\limsup_{n\to+\infty} \,\F(u_n,A) & = & \limsup_{n\to+\infty}\big(\F(u_n,A\setminus S_{u_n})+\F(u_n,A\cap S_{u_n})\big)\\
& \leq & \int_{A\cap S_{u}}\vartheta_{\rm hom}(u^+,u^-,\nu_{u})\,
d\HH^{N-1}+ C\big(\LL^N(A)+|Du|(A\setminus S_u)\big)\,.
\end{eqnarray*}
On the other hand, $\F(\cdot,A)$ is lower semicontinuous with
respect to the strong $L^1(A;\Rb^d)$-convergence, and thus
$\ds\F(u,A)\leq \liminf_{n\to+\infty}\,\F(u_n,A)$ which leads  to
$$\F(u,A)\leq  \int_{A\cap S_{u}}\vartheta_{\rm hom}(u^+,u^-,\nu_{u})\, d\HH^{N-1}+
C\big(\LL^N(A)+|Du|(A\setminus S_u)\big)\,.$$ Since $A$ is
arbitrary, the  above inequality holds for any open set $A\in\A(\O)$
and, by Remark \ref{measconstr}, it also holds if $A$ is any Borel
subset of $\O$. Then taking $A=\O\cap S_u$ yields the desired
inequality.
\end{proof}

\section{The lower bound}

\noindent We adress in this section with the $\G$-$\liminf$
inequality. Using the blow-up method, we follow the approach of
\cite{FM2}, estimating separately the Cantor part and the jump part,
while the bulk part is obtained exactly as in the $W^{1,1}$ analysis, see
\cite[Lemma 5.2]{BM}.

\begin{lemma}\label{lowerboundBV}
For every $u \in BV(\O;\M)$, we have  $\F(u,\O)\geq \F_{\rm hom}
(u)$.
\end{lemma}

\begin{proof}
Let $u \in BV(\O;\M)$ and $\{u_n\} \subset W^{1,1}(\O;\M)$ be such
that $$\F(u,\O)= \lim_{n \to +\infty}\int_\O f
\left(\frac{x}{\e_n},\nabla u_n \right)dx\,.$$ Define the sequence
of nonnegative Radon measures
$$\mu_n:=f \left(\frac{\cdot}{\e_n},\nabla u_n \right)\LL^N
\res\,  \O\,.$$
Up to the extraction of a subsequence, we can assume that
there exists a nonnegative Radon measure $\mu \in \M(\O)$
such that $\mu_n \xrightharpoonup[]{*} \mu$ in $\M(\O)$. By the
Besicovitch Differentiation Theorem,
we
can split $\mu$ into the sum of four mutually singular nonnegative measures
$\mu=\mu^a + \mu^j+\mu^c+\mu^s$ where $\mu^a \ll \mathcal L^N$,
$\mu^j \ll \HH^{N-1}\res\,  S_u$ and $\mu^c \ll |D^c u|$. Since we
have $\mu(\O) \leq \F(u,\O)$, it is enough to check that
\begin{equation}\label{lambda^a}
\frac{d\mu}{d\LL^N}(x_0) \geq Tf_\hom(u(x_0),\nabla u(x_0))\quad
\text{ for }\LL^N\text{-a.e. }x_0 \in \O\,,
\end{equation}
\begin{equation}\label{lambda^c}
\frac{d\mu}{d|D^c u|}(x_0) \geq Tf_\hom^\infty\left(\tilde
u(x_0),\frac{dD^c u}{d|D^c u|}(x_0)\right)\quad \text{ for }|D^c
u|\text{-a.e. }x_0 \in \O\,,
\end{equation}
and
\begin{equation}\label{lambda^j}
\frac{d\mu}{d\HH^{N-1}\res\,  S_u}(x_0) \geq
\vartheta_\hom(u^+(x_0),u^-(x_0),\nu_u(x_0))\quad \text{ for
}\HH^{N-1}\text{-a.e. }x_0 \in S_u\,.
\end{equation}
The proof of (\ref{lambda^a}) follows the one in \cite[Lemma
5.2]{BM} and we shall omit it. The proofs of (\ref{lambda^c}) and
(\ref{lambda^j}) are postponed to the remaining of this subsection.
\end{proof}

\noindent {\bf Proof of \eqref{lambda^c}.} The lower bound on the
density of the Cantor part will be achieved in three steps. We shall
use the blow-up method to reduce the study to constant limits, and
then a truncation argument as in the proof of \cite[Lemma 5.2]{BM},
to replace the starting sequence by a uniformly converging one.
\vskip5pt

{\bf Step 1.} Choose a point $x_0 \in \O$ such that
\begin{equation}\label{cantor2}
\lim_{\rho \to 0^+}\med_{Q(x_0,\rho)}|u(x)-\tilde u(x_0)|\, dx=0\,,
\end{equation}
\begin{equation}\label{cantor3}
A(x_0):=\lim_{\rho \to
0^+}\frac{Du(Q(x_0,\rho))}{|Du|(Q(x_0,\rho))}\in [T_{\tilde
u(x_0)}(\M)]^N\;\text{ is a rank one matrix with }\;|A(x_0)|=1\,,
\end{equation}
\begin{equation}\label{cantor1}
\frac{d\mu}{d|D^cu|}(x_0) \quad \text{exists and is finite and}\quad
\frac{d|Du|}{d|D^cu|}(x_0)=1\,,
\end{equation}
\begin{equation}\label{cantor4}
\lim_{\rho \to 0^+}\frac{|Du|(Q(x_0,\rho))}{\rho^{N-1}}=0 \quad\text{and}\quad
\lim_{\rho \to 0^+}\frac{|Du|(Q(x_0,\rho))}{\rho^N}=+\infty\,,
\end{equation}
\begin{equation}\label{cantor5}
\liminf_{\rho \to 0^+}\,\frac{|Du|(Q(x_0,\rho) \setminus
Q(x_0,\tau\rho))}{|Du|(Q(x_0,\rho))}\leq 1-\tau^N\quad\text{for every $0<\tau<1$}\,.
\end{equation}
It turns out that $|D^c u|$-a.e. $x_0\in\O$ satisfy these
properties. Indeed (\ref{cantor1}) is immediate while
(\ref{cantor2}) is a consequence of the fact that $S_u$ is $|D^c
u|$-negligible. Property (\ref{cantor3}) comes from Alberti Rank One
Theorem together with Lemma \ref{manifold}, (\ref{cantor4}) from
\cite[Proposition~3.92~(a),~(c)]{AFP} and (\ref{cantor5}) from
\cite[Lemma~2.13]{FM2}. Write $A(x_0)=a \otimes \nu$ for some $a \in
\M$ and $\nu \in \SN$. Upon rotating the coordinate axis, one may
assume without loss of generality that $\nu=e_N$. To simplify the
notations, we set $s_0:=\tilde u(x_0)$ and $A_0:=A(x_0)$. \vskip5pt

Fix $t\in(0,1)$ arbitrarily close to $1$, and in view of
\eqref{cantor5}, find a sequence $\rho_k \searrow 0^+$ such
that
\begin{equation}\label{cantor5bis}
\limsup_{k\to+\infty}\,\frac{|Du|(Q(x_0,\rho_k) \setminus
Q(x_0,t\rho_k))}{|Du|(Q(x_0,\rho_k))}\leq 1-t^N\,.
\end{equation}
Now fix $t<\gamma<1$ and set $\gamma':=(1+\gamma)/2$. Using (\ref{cantor1}), we derive
\begin{multline}\label{infmu}
\frac{d\mu}{d|D^cu|}(x_0) =\lim_{k \to +\infty}\frac{\mu(Q(x_0,\rho_k))}{|Du|(Q(x_0,\rho_k))}\geq
\limsup_{k \to +\infty}\,\frac{\mu(\overline{Q(x_0,\gamma'\rho_k)})}{|Du|(Q(x_0,\rho_k))}\geq\\
\geq \limsup_{k\to+\infty}\,\limsup_{n\to+\infty}\,
\frac{1}{|Du|(Q(x_0,\rho_k))}\int_{Q(x_0,\gamma'
\rho_k)}f\bigg(\frac{x}{\e_n},\nabla u_n\bigg)dx\,.
\end{multline}
Arguing as in the proof of \cite[Lemma 5.2]{BM} with minor
modifications, we construct a sequence $\{\bar v_{n}\} \subset
W^{1,\infty}(Q(0,\rho_k);\Rb^d)$ satisfying $\bar v_{n} \to
u(x_0+\cdot)$ in $L^1(Q(0,\rho_k);\Rb^d)$ and
\begin{equation}\label{passuv}
\limsup_{n\to+\infty}\, \int_{Q(x_0,\gamma'
\rho_k)}f\bigg(\frac{x}{\e_n},\nabla u_n\bigg)\, dx\geq
\limsup_{n\to+\infty}\, \int_{Q(0,\gamma
\rho_k)}g\bigg(\frac{x}{\e_n}, \bar v_n, \nabla \bar v_n\bigg)\,
dx\,,
\end{equation}
where $g$ is given by \eqref{defg}. Setting $w_{n,k}(x):=\bar
v_{n}(\rho_k\, x)$, a change of variable together with \eqref{infmu}
and \eqref{passuv} yields
\begin{equation}\label{c1}
\frac{d\mu}{d|D^cu|}(x_0) \geq \limsup_{k \to +\infty} \limsup_{n
\to +\infty}\, \frac{\rho_k^N}{|Du|(Q(x_0,\rho_k))} \int_{\gamma Q}
g\left(\frac{\rho_k\, x}{\e_n},w_{n,k}, \frac{1}{\rho_k} \nabla
w_{n,k}\right)dx\,.
\end{equation}
Then we infer from (\ref{cantor2}) that
\begin{equation}\label{c2}
\lim_{k \to +\infty} \lim_{n \to +\infty}\int_Q|w_{n,k}-s_0|\, dx=0\,,
\end{equation}
and
\begin{multline}\label{c3}
\lim_{k \to +\infty} \lim_{n \to
+\infty}\frac{\rho_k^{N-1}}{|Du|(Q(x_0,\rho_k))}\int_Q
\bigg|w_{n,k}(x)-u(x_0+\rho_k x)\\
- \int_Q \big(w_{n,k}(y)-u(x_0+\rho_k y)\big)\, dy\bigg|\, dx =0\,.
\end{multline}
By \eqref{c1}, \eqref{c2} and (\ref{c3}), we can extract a diagonal sequence $n_k
\to+\infty$ such that  $\d_k:=\e_{n_k}/\rho_k\to 0$,
$w_k:=w_{n_k,k}\to s_0$ in
$L^1(Q;\Rb^d)$,
$$\frac{d\mu}{d|D^cu|}(x_0) \geq \limsup_{k \to +\infty}\,
\frac{\rho_k^N}{|Du|(Q(x_0,\rho_k))} \int_{\gamma Q}
g\left(\frac{x}{\d_k},w_k, \frac{1}{\rho_k} \nabla w_k\right)dx\,,$$
and
\begin{equation}\label{c5}
\lim_{k \to +\infty}\frac{\rho_k^{N-1}}{|Du|(Q(x_0,\rho_k))}\int_Q
\bigg|w_k(x)-u(x_0+\rho_k\, x) - \int_Q \big(w_k(y)-u(x_0+\rho_k\, y)\big)\,
dy\bigg|\,dx =0\,.
\end{equation}
\vskip5pt

{\bf Step 2.} Now we reproduce the truncation argument used in Step
2 of the proof of \cite[Lemma~5.2]{BM} with minor modifications
(make use of (\ref{cantor4}) and \cite[Lemma~2.12]{FM2} instead of
\cite[Lemma~2.6]{FM}, see \cite{FM2} for details). Setting
$a_k:=\int_Q w_k(y)\, dy$, it yields a sequence of cut-off functions
$\{\zeta_k\} \subset \C^\infty_c(\Rb;[0,1])$ such that
$\zeta_k(\tau)=1$ if $|\tau| \leq s_k$, $\zeta_k(\tau)=0$ is
$|\tau|\geq t_k$ for some
$$\|w_k-a_k\|^{1/2}_{L^1(Q;\Rb^d)}<s_k<t_k<\|w_k-a_k\|^{1/3}_{L^1(Q;\Rb^d)}\,,$$
for which $\overline   w_k:=a_k +\zeta_k(|w_k-a_k|)(w_k-a_k)\in W^{1,1}(Q;\Rb^d)$
satisfies
$\overline w_k
\to s_0$ in $L^\infty(Q;\Rb^d)$ and
\begin{equation}\label{c6}
\frac{d\mu}{d|D^cu|}(x_0) \geq \limsup_{k \to +\infty}\,
\frac{\rho_k^N}{|Du|(Q(x_0,\rho_k))} \int_{\gamma Q}
g\left(\frac{x}{\d_k},\overline   w_k, \frac{1}{\rho_k} \nabla
\overline   w_k\right)dx\,.
\end{equation}
In view of the coercivity condition (\ref{pgrowth}), (\ref{cantor1}) and (\ref{c6}),
$$\sup_{k \in \Nb}\,  \frac{\rho_k^{N-1}}{|Du|(Q(x_0,\rho_k))} \int_{\gamma Q} |\nabla \overline w_k|\, dx<+\infty\,.$$
Therefore,  (\ref{moduluscont}), (\ref{c6}) and $\|\overline w_k -s_0\|_{L^\infty(Q;\Rb^d)}\to 0$ lead to
$$\frac{d\mu}{d|D^cu|}(x_0)  \geq \limsup_{k \to +\infty}\,
\frac{\rho_k^N}{|Du|(Q(x_0,\rho_k))} \int_{\gamma Q}
g\left(\frac{x}{\d_k},s_0, \frac{1}{\rho_k} \nabla \overline
w_k\right)dx\,.$$
Next we define the three following sequences for every $x\in Q$,
$$\left\{\begin{array}{l}
\ds \overline  u_k(x):=\frac{\rho_k^{N-1}}{|Du|(Q(x_0,\rho_k))}
\bigg(u(x_0+\rho_k \, x) -\int_Q u(x_0 +\rho_k\, y)\, dy \bigg)\,,\\[0.4cm]
\ds z_k(x):=\frac{\rho_k^{N-1}}{|Du|(Q(x_0,\rho_k))} \big(w_k(x) -a_k \big)\,,\\[0.4cm]
\ds \overline   z_k(x):=\frac{\rho_k^{N-1}}{|Du|(Q(x_0,\rho_k))}
\big(\overline   w_k(x) - a_k \big)\,.
\end{array}\right.$$
As a consequence of (\ref{c5}) we have  $\|z_k - \overline  u_k
\|_{L^1(Q;\Rb^d)} \to 0$, and since
$$\int_Q \overline  u_k(x)\, dx = 0 \quad \text{ and } \quad |D\overline  u_k|(Q)=1\,,$$
it follows that the sequence $\{\overline u_k\}$ is bounded in $BV(Q;\Rb^d)$ and thus  relatively
compact in $L^1(Q;\Rb^d)$. Hence $\{\overline u_k\}$ is equi-integrable, and consequently so
is $\{z_k\}$. Up to a subsequence, $\overline u_k$ converges in $L^1(Q;\Rb^d)$ to some  function $v\in BV(Q;\Rb^d)$, and then $z_k\to v$ in $L^1(Q;\Rb^d)$.
By \cite[Theorem~3.95]{AFP}
the limit $v$ is representable by
$$v(x)=a\, \theta(x_N)$$ for some increasing function $\theta\in BV((-1/2,1/2);\Rb)$ (recall that we assume $A_0=a \otimes e_N$).

By construction, $\overline w_k$ coincides with $w_k$ in the set
$\{|w_k-a_k| \leq s_k\}$. Hence
\begin{multline}\label{ei}
\|\overline   z_k - z_k\|_{L^1(Q;\Rb^d)} =  \frac{\rho_k^{N-1}}{|Du|(Q(x_0,\rho_k))} \int_{\{|w_k-a_k|> s_k\}}|w_k(x)-\overline   w_k(x)|\, dx\leq \\
 \leq \frac{\rho_k^{N-1}}{|Du|(Q(x_0,\rho_k))} \int_{\{|w_k-a_k|> s_k\}}|w_k(x)-a_k|\, dx=  \int_{\{|w_k-a_k|> s_k\}} |z_k(x)|\, dx\,.
\end{multline}
By Chebyshev inequality, we have
\begin{equation}\label{ln}\LL^N(\{|w_k-a_k|> s_k\}) \leq \frac{1}{s_k}\int_Q|w_k(x)-a_k|\, dx
\leq \|w_k-a_k\|^{1/2}_{L^1(Q;\Rb^d)}\to 0\,,
\end{equation}
and thus (\ref{ei}), (\ref{ln}) and the equi-integrability of $\{z_k\}$ imply $\|\overline   z_k -
z_k\|_{L^1(Q;\Rb^d)} \to 0$. Therefore $\overline   z_k
\to v$ in $L^1(Q;\Rb^d)$, and setting
$\a_k:=|Du|(Q(x_0,\rho_k))/\rho_k^N \to +\infty$,
\begin{equation}\label{tilde}
\frac{d\mu}{d|D^cu|}(x_0) \geq \limsup_{k \to +\infty}\,
\frac{1}{\a_k} \int_{\gamma Q} g\left(\frac{x}{\d_k},s_0,\a_k \nabla
\overline   z_k\right)dx\,.
\end{equation}
Using (\ref{grec}) and the positive $1$-homogeneity of the recession function
$g^\infty(y,s,\cdot)$, we infer that
\begin{align*}
 \int_{\gamma Q} \left| \frac{1}{\a_k}\, g\left(\frac{x}{\d_k},s_0,\a_k
\nabla \overline   z_k\right) -
g^\infty\left(\frac{x}{\d_k},s_0, \nabla \overline   z_k\right)\right| dx &\leq \frac{C}{\a_k} \int_{\gamma Q} (1+ \a_k^{1-q} |\nabla \overline z_k|^{1-q})\, dx\\
&\leq C\big(\a_k^{-1} + \a_k^{-q} \|\nabla \overline
z_k\|^{1-q}_{L^1(\gamma Q;\Rb^{d \times N})}\big) \to 0\,,
\end{align*}
where we have used H\"older's inequality and the boundedness of
$\{\nabla \overline   z_k\}$ in $L^1(\gamma Q;\Rb^{d
\times N})$ (which follows from (\ref{pgrowth}) and \eqref{tilde}).
Consequently,
$$\frac{d\mu}{d|D^cu|}(x_0) \geq \limsup_{k \to +\infty}\,
\int_{\gamma Q} g^\infty \left(\frac{x}{\d_k},s_0,\nabla \overline
z_k\right)dx\,.$$
\vskip5pt

{\bf Step 3.} Extend $\theta$ continuously to $\Rb$ by the values of
its traces at $\pm 1/2$. Define $v_k(x)=v_k(x_N):=a \theta * \varrho_k (x_N)$ where
$\varrho_k$ is a sequence of (one dimensional) mollifiers. Then $v_k
\to v$ in $L^1(Q;\Rb^d)$ and thus, since $\overline  u_k - v_k \to
0$ in $L^1(Q;\Rb^d)$, it follows that (up to a subsequence)
\begin{equation}\label{D}
D\overline  u_k(\tau Q) - Dv_k(\tau Q) \to 0
\end{equation}
for $\LL^1$-a.e. $\tau \in (0,1)$. Fix $\tau\in (t,\gamma)$
for which \eqref{D} holds. Since $\|\bar z_k -v_k\|_{L^1(Q;\Rb^d)}\to 0$, one can use a standard cut-off
function argument (see \cite[p. 29--30]{FM2}) to modify the sequence $\{\overline z_k\}$
and  produce a new sequence
$\{\overline\varphi_k\} \subset W^{1,\infty}(\tau Q;\Rb^d)$ satisfying
$\overline \varphi_k \to v$ in $L^1(\tau Q;\Rb^d)$, $\overline
\varphi_k=v_k$ on a neighborhood of $\partial (\tau Q)$ and
\begin{equation}\label{may}
\frac{d\mu}{d|D^cu|}(x_0) \geq \limsup_{k \to +\infty} \int_{\tau Q}
g^\infty \left(\frac{x}{\d_k},s_0,\nabla \overline
\varphi_k\right)dx\,.
\end{equation}
A simple computation shows that
\begin{equation}\label{D2}
D\overline  u_k (\tau Q)=\frac{Du(Q(x_0,\tau
\rho_k))}{|Du|(Q(x_0,\rho_k))}\quad \text{ and }\quad Dv_k(\tau
Q)=\tau^N \,A_k\,,
\end{equation}
where $A_k \in \Rb^{d \times N}$ is the matrix given by
$$A_k:= a \otimes e_N \frac{\theta *\varrho_k (\tau/2)- \theta *\varrho_k (-\tau/2)}{\tau}\,.$$
We observe that $A_k$ is bounded in $k$ since $\theta$ has bounded
variation.

Let $m_k:=[\tau/\delta_k]+1\in \Nb$, and define for $x=(x',x_N)\in \delta_km_k Q$,
$$\varphi_k(x):=\begin{cases}
\overline \varphi_k(x)-A_kx & \text{if $x\in \tau Q$}\,,\\
v_k(x_N)-A_k\, x & \text{if $|x_N|\leq \tau/2$ and $|x'|\geq \tau/2$}\, ,\\
v_k(\tau/2)-A_k(x',\tau/2) & \text{if $x_N\geq \tau/2$}\,,\\
v_k(-\tau/2)-A_k(x',-\tau/2) & \text{if $x_N\leq -\tau/2$}\,.
\end{cases}$$
One may check that $\varphi_k\in W^{1,\infty}(\delta_km_kQ;\Rb^d)$, $\varphi_k$ is  $\delta_km_k$-periodic, and that
\begin{equation}\label{periodization}
 \limsup_{k \to +\infty} \int_{\tau Q}
g^\infty \left(\frac{x}{\d_k},s_0,\nabla \overline
\varphi_k\right)dx= \limsup_{k \to +\infty} \int_{\delta_km_k Q}
g^\infty \left(\frac{x}{\d_k},s_0,A_k+\nabla
\varphi_k\right)dx\,.
\end{equation}
Setting $\phi_k(y):=\tau^{N}\delta_k^{-1}\varphi_k(\delta_k y)$ for $y\in m_k Q$, we have $\phi_k\in W^{1,\infty}_{\#}(m_kQ;\Rb^d)$, and a change of variables yields
\begin{align}
\nonumber\int_{\delta_km_k Q} g^\infty
\left(\frac{x}{\d_k},s_0,A_k+\nabla varphi_k\right)dx&=
\tau^{-N}\delta_k^Nm_k^N\med_{m_k Q} g^\infty
\left(y,s_0,\tau^{N}A_k+\nabla
\phi_k\right)dy \\
\label{compper} &\geq \tau^{-N}\delta_k^Nm_k^N (g^\infty)_{\rm hom}(s_0,\tau^N A_k)\,,
\end{align}
since $(g^\infty)_\hom$ can be computed as  follows
(see Remark \ref{reminfty} and {\it e.g., } \cite[Remark~14.6]{BD}),
\begin{eqnarray*}
(g^\infty)_\hom(s,\xi)
=  \inf \left\{\med_{(0,m)^N} g^\infty(y,s,\xi +\nabla \phi(y))\, dy : m\in \Nb,\,
\phi \in W^{1,\infty}_\#((0,m)^N;\Rb^d) \right\}\,.
\end{eqnarray*}
Gathering \eqref{may}, \eqref{periodization} and \eqref{compper}, we derive
\begin{equation*}
\frac{d\mu}{d|D^cu|}(x_0) \geq\limsup_{k \to +\infty}
\,(g^\infty)_\hom(s_0,\tau^N A_k)\,.
\end{equation*}
In view (\ref{D}), (\ref{D2}), (\ref{cantor5bis}) and (\ref{cantor3}), we have
\begin{multline*}
\limsup_{k \to +\infty}|\tau^N A_k -A_0|=\limsup_{k \to +\infty}|Dv_k(\tau Q) - A_0|=\limsup_{k \to +\infty}|D\overline  u_k(\tau Q)-A_0|=\\
=\limsup_{k \to +\infty}\left| \frac{Du(Q(x_0,\tau \rho_k))}{|Du|(Q(x_0,\rho_k))}-A_0\right|
= \limsup_{k \to +\infty}
\frac{|Du|(Q(x_0,\rho_k)\setminus Q(x_0,\tau\rho_k))}
{|Du|(Q(x_0,\rho_k))}\leq 1- t^{N}\,.
\end{multline*}
By Remark \ref{reminfty},  $(g^\infty)_\hom(s_0,\cdot)$ is Lipschitz continuous, and consequently
$$\frac{d\mu}{d|D^cu|}(x_0) \geq (g^\infty)_\hom(s_0,A_0)-C(1-t^{N})\,.$$
From the arbitrariness of $t$, we finally infer that
$$\frac{d\mu}{d|D^cu|}(x_0) \geq (g^\infty)_\hom(s_0,A_0)\,. $$
Since $s_0 \in \M$ and $A_0 \in [T_{s_0}(\M)]^N$, Remark
\ref{reminfty} and \eqref{remdensinf} yield
$(g^\infty)_\hom(s_0,A_0)= T(f^\infty)_\hom(s_0,A_0)\geq
Tf_\hom^\infty(s_0,A_0)$, and the proof is complete. \prbox
\vskip10pt

\noindent{\bf Proof of \eqref{lambda^j}.} The strategy used in that
part follows the one already used for the bulk and Cantor parts. It
still rests on the blow up method together with the projection
argument in Proposition~\ref{proj}. \vskip5pt

{\bf Step 1.} Let $x_0 \in S_u$ be such that
\begin{equation}\label{jump1}
\lim_{\rho \to 0^+}\med_{Q_{\nu_u(x_0)}^\pm(x_0,\rho)} |u(x)-u^\pm(x_0)|\, dx=0\,,
\end{equation}
where $u^\pm(x_0) \in \M$,
\begin{equation}\label{jump2}
\lim_{\rho \to 0^+}\frac{\HH^{N-1}(S_u \cap Q_{\nu_u(x_0)}(x_0,\rho))}{\rho^{N-1}}=1\,,
\end{equation}
and such that the Radon-Nikod\'ym derivative of $\mu$ with respect
to $\HH^{N-1}\res \, S_u$ exists and is finite. By
Lemma~\ref{manifold}, Theorem 3.78 and Theorem 2.83 (i) in
\cite{AFP} (with cubes instead of balls), it turns out that
$\HH^{N-1}$-a.e.  $x_0\in S_u$ satisfy these properties. Set
$s_0^\pm:=u^\pm(x_0)$,  $\nu_0:=\nu_u(x_0)$.

Up to a further subsequence, we may assume that  $(1+|\nabla u_n|) \LL^N
\res\, \O \xrightharpoonup[]{*} \lambda$ in $\M(\O)$ for some nonnegative
Radon measure $\lambda \in \M(\O)$.  Consider a sequence $\rho_k \searrow 0^+$ satisfying $\mu(\partial
Q_{\nu_0}(x_0,\rho_k))=\lambda(\partial
Q_{\nu_0}(x_0,\rho_k))=0$ for each $k \in \Nb$. Using
(\ref{jump2}) we derive
\begin{multline*}
\frac{d\mu}{d\HH^{N-1}\res \, S_u}(x_0)\lim_{k \to +\infty}\frac{\mu(Q_{\nu_0}(x_0,\rho_k))}{\HH^{N-1}(S_u \cap Q_{\nu_0}(x_0,\rho_k))}
=\lim_{k \to +\infty}\frac{\mu(Q_{\nu_0}(x_0,\rho_k))}{\rho_k^{N-1}}=\\
=  \lim_{k \to +\infty}\lim_{n \to
+\infty}\frac{1}{\rho_k^{N-1}}\int_{Q_{\nu_0}(x_0,\rho_k)}
f\left(\frac{x}{\e_n},\nabla u_n \right)dx\,.
\end{multline*}
Thanks to Theorem \ref{density}, one can assume without loss of
generality that $u_n \in \mathcal D(\O;\M)$ for each $n \in \Nb$.
Arguing exactly as in Step 1 of the proof of \cite[Lemma 5.2]{BM}
(with $Q_{\nu_0}(x_0,\rho_k)$ instead of $Q(x_0,\rho_k)$) we obtain
a sequence $\{v_n\} \subset \mathcal D(Q_{\nu_0}(0,\rho_k);\M)$ such
that $v_n \to u(x_0+\cdot)$ in $L^1(Q_{\nu_0}(0,\rho_k);\Rb^d)$ as
$n \to +\infty$, and
$$\frac{d\mu}{d\HH^{N-1}\res \, S_u}(x_0) \geq \limsup_{k \to +\infty}\,\limsup_{n \to +\infty}\,
\frac{1}{\rho_k^{N-1}}\int_{Q_{\nu_0}(0,\rho_k)}
f\left(\frac{x}{\e_n},\nabla v_n \right)dx$$
(note that the
construction process to obtain $v_n$ from $u_n$ does not affect the
manifold constraint). Changing variables and setting
$w_{n,k}(x)=v_n(\rho_k\, x)$ lead to
$$\frac{d\mu}{d\HH^{N-1}\res \, S_u}(x_0) \geq \limsup_{k \to
+\infty}\,\limsup_{n \to +\infty} \rho_k\int_{Q_{\nu_0}}
f\left(\frac{\rho_k \, x}{\e_n},\frac{1}{\rho_k}\nabla w_{n,k}
\right)dx\,.$$
Defining
$$u_0(x):=\begin{cases}
s_0^+ & \text{ if }  x\cdot \nu_0 > 0\,,\\
s_0^- & \text{ if }  x\cdot \nu_0 \leq 0\,,
\end{cases}$$
we infer from (\ref{jump1}) that
$$\lim_{k \to +\infty}\lim_{n \to +\infty}
\int_{Q_{\nu_0}}|w_{n,k}-u_0|\, dx=0\,.$$ By a standard diagonal
argument, we find a sequence $n_k \nearrow +\infty$ such that
 $\d_k:=\e_{n_k}/\rho_k\to 0$,  $w_k:=w_{n_k,k} \in \mathcal
D(Q_{\nu_0};\M)$ converges to $u_0$ in $L^1(Q_{\nu_0};\Rb^d)$, and
\begin{equation}\label{dhnwj}
\frac{d\mu}{d\HH^{N-1}\res \, S_u}(x_0) \geq \limsup_{k \to +\infty}\,
\rho_k\int_{Q_{\nu_0}}f\left(\frac{x}{\d_k},\frac{1}{\rho_k}\nabla
w_k \right)dx\,.
\end{equation}
According to $(H_4)$ and the positive $1$-homogeneity of $f^\infty(y,\cdot)$, we have
\begin{align}\label{dhnwj2}
\nonumber\int_{Q_{\nu_0}}\left|\rho_k\,
f\left(\frac{x}{\d_k},\frac{1}{\rho_k}\nabla w_k \right) -
f^\infty\left(\frac{x}{\d_k},\nabla w_k \right)\right| dx & \leq  C\rho_k\int_{Q_{\nu_0}} (1 + \rho_k^{q-1}
|\nabla w_k|^{1-q})\, dx\\
&\leq C\left(\rho_k +\rho_k^q \|\nabla
w_k\|^{1-q}_{L^1(Q_{\nu_0};\Rb^{d \times N})}\right)\,,
\end{align}
where we have used H\"older's inequality and $0<q<1$. From
(\ref{dhnwj}) and the coercivity condition $(H_2)$, it follows that
$\{\nabla w_k\}$ is uniformly bounded in $L^1(Q_{\nu_0};\Rb^{d
\times N})$. Gathering (\ref{dhnwj}) and (\ref{dhnwj2}) yields
\begin{equation}\label{j1}
\frac{d\mu}{d\HH^{N-1}\res \, S_u}(x_0) \geq \limsup_{k \to +\infty}\,
\int_{Q_{\nu_0}}f^\infty\left(\frac{x}{\d_k},\nabla w_k \right)dx\,.
\end{equation}
\vskip5pt

{\bf Step 2.} Now it remains to modify the value of $w_k$
on a neighborhood of $\partial Q_{\nu_0}$ in order to get an
admissible test function for the surface energy density. We argue as
in \cite[Lemma~5.2]{AEL}. Using the notations of
Subsection~\ref{sectsurf}, we consider $\g \in
\mathcal G(s_0^+,s_0^-)$, and set
$$\psi_k(x):=\g\left(\frac{x\cdot\nu_0}{\d_k} \right)\,.$$
Using a De Giorgi type slicing argument, we shall modify $w_k$ in
order to get a function which matches $\psi_k$ on $\partial
Q_{\nu_0}$. To this end, define
$$r_k:=\|w_k-\psi_k\|^{1/2}_{L^1(Q_{\nu_0};\Rb^d)}\,, \quad
M_k:=k[1+\|w_k\|_{W^{1,1}(Q_{\nu_0};\Rb^d)}
+\|\psi_k\|_{W^{1,1}(Q_{\nu_0};\Rb^d)}]\,, \quad
\ell_k:=\frac{r_k}{M_k}\,.$$
Since $\psi_k$ and $w_k$ converge to $u_0$ in
$L^1(Q_{\nu_0};\Rb^d)$, we have $r_k \to 0$, and one may assume that
$0<r_k<1$. Set
$$Q^{(i)}_k:=(1-r_k+i\, \ell_k) Q_{\nu_0} \quad \text{ for }i=0,\ldots,M_k\,.$$
For every $i \in \{1,\ldots,M_k\}$, consider a cut-off function
$\varphi^{(i)}_k \in \C^\infty_c(Q^{(i)}_k;[0,1])$ satisfying
$\varphi^{(i)}_k=1$ on $Q^{(i-1)}_k$ and $|\nabla
\varphi^{(i)}_k|\leq c/\ell_k$. Define
$$z^{(i)}_k:=\varphi^{(i)}_k
w_k + (1-\varphi^{(i)}_k)\psi_k\in W^{1,1}(Q_{\nu_0};\Rb^d)\,,$$ so
that $z_k^{(i)}=w_k$ in $Q^{(i-1)}_k$, and $z_k^{(i)}=\psi_k$ in
$Q_{\nu_0} \setminus Q^{(i)}_k$. Since $z^{(i)}_k$ is smooth outside
a finite union of sets contained in some $(N-2)$-dimensional
submanifolds and $z^{(i)}_k(x) \in {\rm co}(\M)$ for a.e. $x\in
Q_{\nu_0}$, one can apply Proposition \ref{proj} to obtain new
functions $\hat z^{(i)}_k \in W^{1,1}(Q_{\nu_0};\M)$ such that $\hat
z^{(i)}_k= z^{(i)}_k$ on $(Q_{\nu_0} \setminus Q^{(i)}_k) \cup
Q^{(i-1)}_k$, and
\begin{align*}
\int_{Q^{(i)}_k \setminus Q^{(i-1)}_k}|\nabla \hat z^{(i)}_k|\, dx
&\leq C_\star \int_{Q^{(i)}_k \setminus Q^{(i-1)}_k}|\nabla z^{(i)}_k|\,
dx\\
&\leq C_\star \int_{Q^{(i)}_k \setminus Q^{(i-1)}_k}\left(|\nabla
w_k|+|\nabla \psi_k| + \frac{1}{\ell_k}|w_k-\psi_k| \right)\,
dx\,.
\end{align*}
In particular $\hat z^{(i)}_k \in \mathcal
B_{\d_k}(s_0^+,s_0^-,\nu_0)$, and by the growth condition
(\ref{finfty1gc}),
\begin{multline*}
\int_{Q_{\nu_0}}f^\infty\left(\frac{x}{\d_k},\nabla \hat
z^{(i)}_k\right)dx \leq
\int_{Q_{\nu_0}}f^\infty\left(\frac{x}{\d_k},\nabla  w_k
\right)dx+C\int_{Q_{\nu_0} \setminus Q^{(i)}_k} |\nabla \psi_k|\, dx\,+\\
+C\int_{Q^{(i)}_k\setminus Q^{(i-1)}_k} \left(|\nabla w_k|+|\nabla
\psi_k| + \frac{1}{\ell_k}|w_k-\psi_k| \right)\, dx\,.
\end{multline*}
Summing up over all $i=1,\ldots,M_k$ and dividing by $M_k$, we get
that
\begin{multline*}
\frac{1}{M_k}\sum_{i=1}^{M_k}\int_{Q_{\nu_0}}f^\infty\left(\frac{x}{\d_k},\nabla
\hat z^{(i)}_k\right)dx  \leq
\int_{Q_{\nu_0}}f^\infty\left(\frac{x}{\d_k},\nabla  w_k
\right)dx\,+\\
+ C\int_{Q_{\nu_0} \setminus Q^{(0)}_k} |\nabla \psi_k|\, dx
+\frac{C}{k}+C\|w_k-\psi_k\|^{1/2}_{L^1(Q_{\nu_0};\Rb^d)}\,.
\end{multline*}
Since $$\int_{Q_{\nu_0} \setminus Q^{(0)}_k}|\nabla \psi_k|\, dx
\leq \mathbf d_{\M}(s_0^+,s_0^-) \HH^{N-1}((Q_{\nu_0} \setminus
Q^{(0)}_k) \cap \{x\cdot \nu_0=0\}) \to 0$$ as $k \to +\infty$,
there exists a sequence $\eta_k \to 0^+$ such that
$$\frac{1}{M_k}\sum_{i=1}^{M_k}\int_{Q_{\nu_0}}f^\infty\left(\frac{x}{\d_k},\nabla
\hat z^{(i)}_k\right)dx \leq
\int_{Q_{\nu_0}}f^\infty\left(\frac{x}{\d_k},\nabla  w_k \right)dx
+\eta_k\,.$$
Hence, for each $k \in \Nb$ we can find some index
$i_k \in \{1,\ldots,M_k\}$ satisfying
\begin{equation}\label{j11}
\int_{Q_{\nu_0}}f^\infty\left(\frac{x}{\d_k},\nabla \hat
z_k^{(i_k)}\right)dx \leq
\int_{Q_{\nu_0}}f^\infty\left(\frac{x}{\d_k},\nabla  w_k \right)dx
+\eta_k\,.
\end{equation}
Gathering (\ref{j1}) and (\ref{j11}), we obtain that
$$\frac{d\mu}{d\HH^{N-1}\res \, S_u}(x_0) \geq \limsup_{k \to
+\infty} \int_{Q_{\nu_0}}f^\infty\left(\frac{x}{\d_k},\nabla \hat
z_k^{(i_k)} \right)dx\,.$$
Since $\hat z_k^{(i_k)} \in \mathcal
B_{\d_k}(s_0^+,s_0^-,\nu_0)$, we infer from Proposition
\ref{limitsurfenerg}, Proposition \ref{limsurf2} and  \eqref{idsurfen} that
$$\frac{d\mu}{d\HH^{N-1}\res \, S_u}(x_0) \geq
\vartheta_\hom(s_0^+,s_0^-,\nu_0)\,,$$
which completes the proof.
\prbox

\subsection{Proof of Theorem \ref{babmil2}}

\noindent {\bf Proof of Theorem \ref{babmil2}. } In view of $(H_2)$
and the closure of the pointwise constraint under strong
$L^1$-convergence, $\F(u)<+\infty$ implies $u \in BV(\O;\M)$. In
view of \eqref{decompupbd},  Lemma \ref{upperboundBV},  Corollary
\ref{upbdjp} and Lemma~\ref{lowerboundBV}, the subsequence
$\{\F_{\e_{n}}\}$ $\Gamma$-converges to $\F_\hom$ in
$L^1(\O;\Rb^d)$. Since the $\G$-limit does not depend on the
particular choice of the subsequence, we get in light of
\cite[Proposition~8.3]{DM} that the whole sequence $\G$-converges.
\prbox

\section{Appendix}
\vskip5pt

\noindent We present in this appendix a relaxation result already proved in
\cite{AEL} for $\M=\Sd$, and in \cite{Mucci} for isotropic
integrands. The proof can be obtained
following the one of \cite[Theorem~3.1]{AEL}
replacing the standard projection on the
sphere (used in Lemma 5.2, Proposition~6.2 and Lemma~6.4 of
\cite{AEL}) by the projection on $\M$ of \cite{HL} as in Proposition
\ref{proj}. Since we only make use of the upper bound on the diffuse part,
we will just enlight the differences in the main steps leading to it.
\vskip5pt

Assume that $\M$ is a smooth, compact and connected submanifold of
$\Rb^d$ without boundary, and let $f : \O \times \Rb^d \times \Rb^{d
\times N} \to [0,+\infty)$ be a continous function satisfying:

\begin{itemize}
\item[$(H_1')$] $f$ is tangentially quasiconvex, {\it i.e.}, for all $x \in \O$, all $s \in \M$ and all $\xi \in [T_s(\M)]^N$,
$$f(x,s,\xi) \leq \int_Q f(x,s,\xi + \nabla \varphi(y))\, dy \quad \text{for every $\varphi \in W^{1,\infty}_0(Q;T_s(\M))\,$;}$$

\item[$(H_2')$] there exist $\a>0$ and $\b>0$ such that
$$\a |\xi| \leq f(x,s,\xi) \leq \b(1+|\xi|) \quad \text{ for every
}(x,s,\xi) \in \O \times \Rb^d \times \Rb^{d \times N}\,;$$

\item[$(H_3')$] for every compact set $K \subset \O$, there exists a
continuous function $\omega : [0,+\infty) \to [0,+\infty)$
satisfying $\omega(0)=0$ and
$$|f(x,s,\xi) - f(x',s',\xi)| \leq \omega(|x-x'| + |s-s'|)
(1+|\xi|)$$
for every $x$, $x' \in \O$, $s$, $s' \in \Rb^d$ and
$\xi \in \Rb^{d \times N}$;

\item[$(H_4')$] there exist $C>0$ and $q \in (0,1)$ such that
$$|f(x,s,\xi) - f^\infty(x,s,\xi)| \leq C(1+|\xi|^{1-q}), \quad \text{ for every
}(x,s,\xi) \in \O \times \Rb^d \times \Rb^{d \times N}\,,$$ where
$f^\infty : \O \times \Rb^d \times \Rb^{d \times N} \to [0,+\infty)$
is the recession function of $f$ defined by
$$f^\infty(x,s,\xi):=\limsup_{t \to +\infty} \frac{f(x,s,t\xi)}{t}\,
.$$
\end{itemize}

Consider the functional $F:L^1(\O;\Rb^d) \to [0,+\infty]$ given by
$$F(u):=\left\{
\begin{array}{ll}
\ds \int_\O f(x,u,\nabla u)\, dx & \text{ if }u \in
W^{1,1}(\O;\M),\\[0.4cm]
+\infty & \text{ otherwise}, \end{array}\right.$$ and its relaxation
for the strong $L^1(\O;\Rb^d)$-topology $\overline F:L^1(\O;\Rb^d)
\to [0,+\infty]$ defined by
$$\overline F(u):=\inf_{\{u_n\}} \left\{ \liminf_{n \to +\infty}
F(u_n) : u_n \to u \text{ in }L^1(\O;\Rb^d)\right\}\,.$$ Then the
following integral representation result holds:

\begin{theorem} \label{relax}
Let $\M$ be a smooth compact and connected submanifold of $\Rb^d$
without boundary, and let $f:\Rb^N \times \Rb^{d \times N} \to
[0,+\infty)$ be a continuous function satisfying $(H_1')$ to
$(H_4')$. Then for every $u \in L^1(\O;\Rb^d)$,
\begin{equation}\label{reprel}
\overline F(u)= \begin{cases} \ds
\begin{multlined}[8.5cm]
\,\int_\O f(x,u,\nabla u)dx +
\int_{\O\cap S_u}K(x,u^+,u^-,\nu_u)d\HH^{N-1}\,+ \\[-15pt]
+ \int_\O f^\infty\bigg(x,\tilde u,\frac{dD^cu}{d|D^cu|}\bigg)\,
d|D^cu|
\end{multlined}
& \text{\it if }\,u \in BV(\O;\M)\,,\\
& \\
\,+\infty & \text{\it otherwise}\,,
\end{cases}
\end{equation}
where for every $(x,a,b,\nu) \in \O \times \M \times \M \times \SN$,
\begin{multline*}
K(x,a,b,\nu) := \inf_\varphi \bigg\{\int_{Q_\nu}
f^\infty(x,\varphi(y),\nabla \varphi(y))\, dy : \varphi \in
W^{1,1}(Q_\nu;\M),\; \varphi=a \text{ on } \{x\cdot \nu=1/2\},\\
\varphi=b \text{ on }\{x\cdot \nu=-1/2\} \text{ {\it and} } \varphi \text{
\it is $1$-periodic in the }\nu_2,\ldots,\nu_{N} \text{ directions}
\bigg\}\,,
\end{multline*}
$\{\nu,\nu_2,\ldots,\nu_N\}$ forms any orthonormal basis of $\Rb^N$,
and $Q_\nu$ stands for the open unit cube in $\Rb^N$ centered at the
origin associated to this basis.
\end{theorem}

\noindent{\bf Sketch of the Proof.} The proof of the lower bound
``$\geq$" in \eqref{reprel} can be obtained as in \cite[Lemma~5.2]{BM} and Lemma \ref{lowerboundBV} using standard techniques to
handle with the dependence on the space variable. The lower bounds
for the bulk and Cantor parts rely on the construction of a suitable
function $\tilde f: \O \times \Rb^d \times \Rb^{d \times N} \to
[0,+\infty)$ replacing $f$ as we already pursued in Section
\ref{thbe}. On the other hand, the jump part rests on the projection
on $\M$ of \cite{HL} as in Proposition \ref{proj} instead of the
standard projection on the sphere used in
\cite[Proposition~5.2]{AEL}. \vskip5pt

\noindent
To obtain the upper bound, we localize as usual  the functionals setting for every $u \in
L^1(\O;\Rb^d)$ and $A \in \A(\O)$,
$$F(u,A):=\begin{cases}
\ds \int_A f(x,u,\nabla u)\, dx & \text{ if }u \in
W^{1,1}(A;\M)\,,\\
+\infty & \text{ otherwise}\,,
\end{cases}$$
$$\overline F(u,A):=\inf_{\{u_n\}} \left\{ \liminf_{n \to +\infty}
F(u_n,A) : u_n \to u \text{ in }L^1(A;\Rb^d)\right\}\,.$$
Arguing as in the proof of Lemma \ref{measbis}, we obtain that for every $u \in BV(\O;\M)$, the set function
$\overline F(u,\cdot)$ is the restriction to $\A(\O)$ of a Radon
measure absolutely continuous with respect to $\LL^N+|Du|$. Hence it uniquely extends into a Radon measure on $\Omega$ (see Remark \ref{measconstr}),
and it suffices to prove that for any $u\in BV(\O;\M)$,
\begin{equation}\label{jppartrel}
\overline F(u,\Omega\cap S_u)\leq \int_{\Omega \cap S_u}K(x,u^+,u^-,\nu_u)\, d\HH^{N-1}\,,
\end{equation}
\begin{equation}\label{contpartrel}
\frac{d\overline F(u,\cdot)}{d \LL^N}(x_0)\leq f(x_0,u(x_0),\nabla u(x_0))\quad \text{for $\LL^N$-a.e. $x_0\in \Omega$}\,,
\end{equation}
\begin{equation}\label{cantpartrel}
\frac{d\overline F(u,\cdot)}{d |D^cu|}(x_0)\leq f^\infty\bigg(x_0,\tilde u(x_0),\frac{dD^c u}{d|D^cu|}(x_0)\bigg)\quad \text{for $|D^cu|$-a.e. $x_0\in \Omega$}\,,
\end{equation}
\vskip5pt

\noindent{\it Proof of \eqref{jppartrel}.} Concerning the jump part, one can proceed  as in \cite[Lemma~6.5]{AEL}.
A slight difference lies in the third step of its proof where one needs to approximate in energy
an arbitrary  $u\in BV(\O;\M)$ by a sequence $\{u_n\}\subset BV(\O;\M)$ such that for each $n$, $u_n$ assumes a finite number of values.
This can be performed as in the proof of Corollary \ref{upbdjp} using the regularity properties of $K$ stated in \cite[Lemma~4.1]{AEL} for $\M=\mathbb{S}^{d-1}$.
\vskip5pt

\noindent{\it Proof of \eqref{contpartrel}.} Let $x_0 \in \O$ be a Lebesgue
point for $u$ and $\nabla u$ such that $u(x_0) \in \M$, $\nabla u(x_0)
\in [T_{u(x_0)}(\M)]^N$,
$$\lim_{\rho \to 0^+} \med_{Q(x_0,\rho)} |u(x) - u(x_0)|(1+|\nabla
u(x)|)\, dx=0\,,\quad \lim_{\rho \to 0^+}\frac{|D^s
u|(Q(x_0,\rho))}{\rho^N}=0\,,$$ and
$$\frac{d |Du|}{d\LL^N}(x_0) \quad \text{ and }\quad \frac{d\overline
F(u,\cdot)}{d\LL^N}(x_0)$$ exist and are finite. Note that
$\LL^N$-a.e. $x_0 \in \O$ satisfy these properties. We select a sequence $\rho_k
\searrow 0^+$  such that $Q(x_0,2\rho_k)\subset \O$ and $|Du|(\partial Q(x_0,\rho_k)) =0$ for
each $k \in \Nb$. Next consider a sequence of standard mollifiers
$\{\varrho_n\}$, and  define $u_n :=\varrho_n * u \in
W^{1,1}(Q(x_0,\rho_k);\Rb^d) \cap \C^\infty(Q(x_0,\rho_k);\Rb^d)$. In the sequel, we shall argue
as in the proof of Proposition \ref{proj} and we refer to it for the notation. Fix $\d>0$ small
enough such that $\pi:\Rb^d\setminus X\to \M$ is smooth
in the $\d$-neighborhood of $\M$. 
Since $u_n$ takes its values in ${\rm co}(\M)$, we can reproduce the proof of Proposition \ref{proj}  to find $a_n^k \in\Rb^d$ with $|a_n^k|<\delta/4$
such that setting $p_n^k:=(\pi_{a_n^k}|_{\M})^{-1}
\circ \pi_{a_n^k}$, $w_n^k:=p_n^k\circ u_n \in W^{1,1}(Q(x_0,\rho_k);\M)$ and
\begin{equation}\label{Ank}
\int_{A_n^k} |\nabla w_n^k|\, dx \leq C_* \int_{A_n^k}|\nabla u_n|\,
dx\,,
\end{equation}
where $A_n^k$ denotes the open set $A_n^k:=\big\{x \in Q(x_0,\rho_k) : {\rm dist}(u_n(x),\M) >\d /2\big\}$. 
Furthermore, since $ \pi$ is smooth in the $\d$-neighborhood
of $\M$ and $|a_n^k|<\d/4$,
there exists a constant $C_\d>0$ independent
of $n$ and $k$ such that
\begin{equation}\label{dn}
|\nabla^2 p_n^k(s)|+|\nabla p_n^k(s)| \leq C_\d \text{ for every $s \in \Rb^d$ satisfying
$\dist(s,\M)\leq \d/2$}\,,
\end{equation}
and consequently,
\begin{equation}\label{cAnk}
|\nabla w_n^k| \leq C_\d |\nabla u_n| \quad \text{$\LL^N$-a.e. in
}Q(x_0,\rho_k) \setminus A_n^k\,.
\end{equation}
Since $u(x) \in \M$ for $\LL^N$-a.e. $x \in \O$, it follows that
$$\LL^N(A_n^k) \leq
\frac{2}{\d} \int_{Q(x_0,\rho_k)} \dist(u_n,\M)\, dx \leq
\frac{2}{\d} \int_{Q(x_0,\rho_k)} |u_n-u|\, dx \xrightarrow[n \to
+\infty]{} 0\,,$$
and then (\ref{dn}) yields
\begin{multline*}
\int_{Q(x_0,\rho_k)}|w_n^k - u|\, dx  =  \int_{
A_n^k}|w_n^k - u|\, dx + \int_{Q(x_0,\rho_k)
\setminus A_n^k}|p_n^k(u_n) - p_n^k(u)|\, dx\leq\\
 \leq  {\rm diam}(\M) \LL^N(A_n^k) + C_\d
\int_{Q(x_0,\rho_k)}|u_n-u|\, dx \xrightarrow[n \to +\infty]{} 0\,.
\end{multline*}
Hence $w_n^k \to u$ in $L^1(Q(x_0,\rho_k);\Rb^d)$ as $n \to +\infty$
so that we are allowed to take $w_n^k$ as competitor, {\it i.e.},
$$\overline F(u,Q(x_0,\rho_k)) \leq \liminf_{n \to +\infty}
\int_{Q(x_0,\rho_k)}f(x,w_n^k,\nabla w_n^k)\, dx\,.$$
At this stage we can argue exactly
as in \cite[Lemma~6.4]{AEL} to prove that for any $\eta>0$
there exists $\lambda=\lambda(\eta)>0$ such that
\begin{multline}\label{1numb}
\overline F(u,Q(x_0,\rho_k))  \leq  \liminf_{n \to
+\infty}\bigg\{\int_{Q(x_0,\rho_k)}f(x_0,u(x_0),\nabla u_n)\,
dx+C\int_{Q(x_0,\rho_k)}|\nabla u_n - \nabla w_n^k|\, dx\, +\\
+ C(\eta +\lambda \rho_k) \int_{Q(x_0,\rho_k)}(1+|\nabla u_n|)\, dx
+C\lambda\int_{Q(x_0,\rho_k)}|w_n^k-u(x_0)|(1+|\nabla
w_n^k|)\, dx \bigg\}\,.
\end{multline}
The first and third term in the right handside of \eqref{1numb} can be treated as in the proof of \cite[Theorem~2.16]{FM2}. 
Concerning the remaining terms, we proceed as follows. 
Using
(\ref{Ank}), (\ref{dn}) and (\ref{cAnk}), we get that
\begin{multline}\label{2numb}
\int_{Q(x_0,\rho_k)}|w_n^k-u(x_0)||\nabla w_n^k|\, dx
\leq{\rm diam}(\M) \int_{A_n^k}|\nabla w_n^k|\, dx\,+\\
+ \int_{Q(x_0,\rho_k)
\setminus A_n^k}|p_n^k(u_n) - p_n^k(u(x_0))| |\nabla w_n^k| \, dx
 \leq C\int_{A_n^k}|\nabla u_n|\, dx\,+\\
+ C_\d \int_{Q(x_0,\rho_k) \setminus A_n^k}|u_n - u(x_0)| |\nabla
u_n|
\,dx 
\leq C_\d \int_{Q(x_0,\rho_k)}|u_n - u(x_0)| |\nabla u_n|\, dx\,,
\end{multline}
where  $C_\d>0$ still denotes some constant depending on $\d$ but independent of $k$ and $n$. Arguing
in a similar way, we also derive
\begin{equation}\label{3}
\int_{Q(x_0,\rho_k)}|\nabla u_n - \nabla w_n^k|\, dx \leq C_\delta
\int_{Q(x_0,\rho_k)} |u_n -u(x_0)| |\nabla u_n|\, dx+
\int_{Q(x_0,\rho_k) \setminus A_n^k} |L_n^k \nabla u_n|\, dx\,,
\end{equation}
where $L_n^k:={\rm Id} - \nabla p_n^k(u(x_0)) \in {\rm
Lin}(\Rb^{d\times d},\Rb^{d \times d})$. 
Gathering \eqref{1numb}, \eqref{2numb} and (\ref{3}) we finally
obtain that
\begin{multline}\label{4}
\overline F(u,Q(x_0,\rho_k))  \leq  \liminf_{n \to
+\infty}\bigg\{\int_{Q(x_0,\rho_k)}f(x_0,u(x_0),\nabla u_n)\,
dx+C\int_{Q(x_0,\rho_k) \setminus A_n^k} |L_n^k \nabla u_n|\, dx\,+\\
+C(\eta +\lambda \rho_k) \int_{Q(x_0,\rho_k)}(1+|\nabla u_n|)\, dx
+C_\d\lambda\int_{Q(x_0,\rho_k)}|u_n-u(x_0)|(1+|\nabla u_n|)\,
dx \bigg\}\,.
\end{multline}
Now we can follow the argument in  \cite[Lemma~6.4]{AEL} to conclude  that
$$\frac{d\overline F(u,\cdot)}{d\LL^N}(x_0) \leq f(x_0,u(x_0),\nabla
u(x_0))\,,$$
which completes the proof of \eqref{contpartrel}.
\vskip5pt

\noindent {\it Proof of \eqref{cantpartrel}.} Once again the proof
parallels the one in \cite[Lemma~6.4]{AEL}. We first proceed as in
the previous reasoning leading to \eqref{4}. Then we can exactly
follow the argument of \cite[Lemma~6.4]{AEL} to obtain
\eqref{cantpartrel}. \prbox

\vskip15pt

\noindent{\bf Acknowledgement. }The authors wish to thank Roberto
Alicandro, Pierre Bousquet, Giovanni Leoni and Domenico Mucci for
several interesting discussions on the subject. This work was initiated while
V. Millot was visiting the department of {\it Functional Analysis and Applications} at S.I.S.S.A.,
he thanks G. Dal Maso  and the whole department for
the warm hospitality. The research of
J.-F. Babadjian was partially supported by the Marie Curie Research
Training Network MRTN-CT-2004-505226 ``Multi-scale modelling and
characterisation for phase transformations in advanced materials''
(MULTIMAT). V. Millot was partially supported by  the Center for
Nonlinear Analysis (CNA) under the National Science Fundation Grant
No. 0405343.


\begin{thebibliography}{99}

{\footnotesize

\bibitem{Al} {\sc G. Alberti}: Rank-one property for derivatives of functions with bounded variation,
{\it Proc. Royal Soc. Edinburgh Sect. A} {\bf 123} (1993), 239--274.

\bibitem{AEL} {\sc R. Alicandro, A. C. Esposito \& C. Leone}:
Relaxation in $BV$ of integral functionals defined on Sobolev
functions with values in the unit sphere, {\it J. Conv. Anal.} {\bf
14} (2007), 69--98.

\bibitem{AL} {\sc R. Alicandro \& C. Leone}: 3D-2D asymptotic
analysis for micromagnetic energies, {\it ESAIM Cont. Optim. Calc.
Var.} {\bf 6} (2001), 489--498.

\bibitem{AB} {\sc L. Ambrosio \& A. Braides}: Functionals defined on
partitions in sets of finite perimeter I: integral representation
and $\Gamma$-convergence, {\it J. Math. Pures Appl.} {\bf 69}
(1990), 285--306.

\bibitem{ADM} {\sc L. Ambrosio \& G. Dal Maso}: On the relaxation in
$BV(\Omega;\mathbb{R}^m)$ of quasiconvex integrals, {\it J. Funct.
Anal.} {\bf 109} (1992), 76--97.

\bibitem{AFP} {\sc L. Ambrosio, N. Fusco \& D. Pallara}: {\it
Functions of bounded variation and free discontinuity problems},
Oxford University Press (2000).

\bibitem{AMT} {\sc L. Ambrosio, S. Mortola \& V.M. Tortorelli}: Functionals with linear growth defined on vector valued $BV$ functions, {\it
J. Math. Pures Appl.} {\bf 70} (1991), 269--323.

\bibitem{AmPal} {\sc L. Ambrosio \& D. Pallara}: Integral representation of relaxed functionals on $BV(\Rb^n;\Rb^k)$ and polyhedral approximation, {\it Indiana Univ. Math. J.} {\bf 42} (1993), 295--321.

\bibitem{BM} {\sc J.-F. Babadjian \& V. Millot}: Homogenization of variational problems in manifold
valued Sobolev spaces, preprint (2008).

\bibitem{B} {\sc F. B\'ethuel}: The approximation problem for
Sobolev maps between two manifolds, {\it Acta Math.} {\bf 167}
(1991), 153--206.

\bibitem{BBC} \textsc{F. B\'ethuel, H. Br\'ezis \& J.M. Coron}: Relaxed energies for harmonic
maps, in {\it Variational methods} (Paris, 1988), 37--52. Progress
in Nonlinear Differential Equations and Their Applications {\bf 4},
Birkh{\"a}user, 1990.

\bibitem{BZ} {\sc F. B\'ethuel \& X. Zheng}: Density of smooth
functions between two manifolds in Sobolev spaces, {\it J. Funct.
Anal.} {\bf 80} (1988), 60--75.

\bibitem{Bouch} {\sc G. Bouchitt\'e}: Convergence et relaxation de fonctionnelles du calcul des
variations \`a croissance lin\'eaire. Application \`a l'homog\'en\'eisation en plasticit\'e, {\it Ann. Fac. Sci. Univ. Toulouse} {\bf 8}
(1986), 7--36.

\bibitem{BFF} {\sc G. Bouchitt\'e, I. Fonseca \& L. Mascarenhas}: A global method for relaxation,
{\it Arch. Rational Mech. Anal.} {\bf 145} (1998), 51--98.

\bibitem{Br} {\sc A. Braides}: Homogenization of some almost periodic coercive
functional, {\it Rend. Accad. Naz. Sci. XL.} {\bf 103} (1985),
313--322.

\bibitem{BD} {\sc A. Braides \& A. Defranceschi}: {\it Homogenization of multiple integrals},
Oxford Lecture Series in Mathematics and its Applications {\bf 12},
Oxford University Press, New York (1998).

\bibitem{BDV} {\sc A. Braides, A. Defranceschi  \& E. Vitali}:
Homogenization of free discontinuity problems, {\it Arch. Rational
Mech. Anal.} {\bf 135} (1996), 297--356.

\bibitem{BCL} \textsc{H. Br\'ezis, J.M. Coron \& E.H. Lieb}: Harmonic maps with defects,
{\it Comm. Math. Phys.} {\bf107} (1986), 649--705.

\bibitem{DFMT} {\sc B. Dacorogna, I. Fonseca, J. Mal\'y \& K.
Trivisa}: Manifold constrained variational problems, {\it Calc. Var.
Part. Diff. Eq.} {\bf 9} (1999), 185--206.

\bibitem{DMsc} {\sc G. Dal Maso}: Integral representation on $BV(\Omega)$ of $\Gamma$-limits of variational integrals, {\it Manuscripta Math.}
{\bf 30}  (1980), 387--416.

\bibitem{DM} {\sc G. Dal Maso}: {\it An Introdution to $\Gamma$-convergence}, Birkh\"auser, Boston (1993).

\bibitem{DAG} {\sc R. De Arcangelis \& G. Gargiulo}: Homogenization
of integral functionals with linear growth defined on vector-valued
functions, {\it NoDEA } {\bf 2} (1995), 371--416.

\bibitem{Dem} {\sc F. Demengel}: On some spaces of functions with bounded derivatives between manifolds,
{\it Differential Integral Equations} {\bf 9} (1996), 173--185.

\bibitem{Federer} {\sc H. Federer}: {\it Geometric measure theory},
Springer-Verlag (1969).

\bibitem{FF} {\sc H. Federer \& W.H. Fleming}: Normal and integral currents, {\it Ann. Math.} {\bf 72} (1960), 458--520.

\bibitem{FM} {\sc I. Fonseca \& S. M\"uller}: Quasiconvex integrands
and lower semicontinuity in $L^1$, {\it SIAM J. Math. Anal.} {\bf
23} (1992), 1081--1098.

\bibitem{FM2} {\sc I. Fonseca \& S. M\"uller}: Relaxation of quasiconvex
functionals in $BV(\Omega;\mathbb{R}^p)$ for integrands
$f(x,u,\nabla u)$, {\it Arch. Rational Mech. Anal.} {\bf 123}
(1993), 1--49.

\bibitem{FR} {\sc I. Fonseca \& P. Rybka}: Relaxation of multiple integrals in space $BV(\Omega;\Rb^p)$, {\it Proceedings Roy. Soc. Ed.} {\bf 121A} (1992), 321--348.

\bibitem{GMSbv} {\sc M. Giaquinta, L. Modica \& J. Sou\v{c}ek}: Functionals with linear growth in the calculus of variations, {\it Comment.
Math. Univ. Carolinae} {\bf 20} (1979), 143--172.

\bibitem{GMS} {\sc M. Giaquinta, L. Modica \& J. Sou\v{c}ek}: {\it
Cartesian currents in the calculus of variations}, Modern surveys in
Mathematics {\bf 37-38}, Springer-Verlag, Berlin (1998).

\bibitem{GM} {\sc M. Giaquinta \& D. Mucci}: The BV-energy of maps into a manifold:
relaxation and density results, {\it Ann. Scuola Norm.
Sup. Pisa. Cl. (5)} {\bf 5} (2006), 483--548.

\bibitem{GM1} {\sc M. Giaquinta \& D. Mucci}: Relaxation results for
a class of functionals with linear growth defined on manifold
contrained mappings, to appear in {\it J. Convex Anal.}

\bibitem{Serr} {\sc C. Goffman \& J. Serrin}: Sublinear functions of
measures and variational integrals, {\it Duke Math. J.} {\bf 31} (1964), 159--178.

\bibitem{HKL} {\sc R. Hardt, D. Kinderlehrer \& F. H. Lin}: Stable
defects of minimizers of constrained variational
principles, {\it Ann. Inst. Henri Poincar\'e Anal. Non Lin\'eaire}
{\bf 5} (1986), 297--322.

\bibitem{HL} {\sc R. Hardt \& F. H. Lin}: Mappings minimizing the
$L^p$ norm of the gradient, {\it Comm. Pure Appl. Math.} {\bf 40}
(1987), 555--588.

\bibitem{Mar} {\sc P. Marcellini}: {Periodic solutions and homogenization
of nonlinear variational problems}, {\it Ann. Mat.
Pura  Appl. (4)}  {\bf 117} (1978), 139--152.

\bibitem{Mucci} {\sc D. Mucci}: Relaxation of isotropic functionals
with linear growth defined on manifold values constrained Sobolev
mappings, to appear in {\it ESAIM Cont. Optim. Calc. Var.}

\bibitem{M} {\sc S. M\"{u}ller}: Homogenization of nonconvex integral
functionals and cellular elastic materials, {\it Arch. Rational
Mech. Anal.} {\bf 99} (1987), 189--212.

}
\end{thebibliography}
\end{document}